\documentclass[11pt]{amsart}

\usepackage{amsmath}

\usepackage{amssymb}

\usepackage{amsfonts}

\usepackage{amsthm}

\usepackage{enumerate}

 \usepackage{esint}
 
 \usepackage{tikz}
 
 \usetikzlibrary{patterns}
 
 \usepackage{pgfplots}

 \usepackage{subcaption}
 \usepackage{graphicx}
 \usepackage{lipsum}

\usepackage{graphicx}

\usepackage{verbatim}

\newtheorem{innercustomprop}{Proposition}
\newenvironment{customprop}[1]
  {\renewcommand\theinnercustomprop{#1}\innercustomprop}
  {\endinnercustomprop}
  
  \newtheorem{innercustomlemma}{Lemma}
\newenvironment{customlemma}[1]
  {\renewcommand\theinnercustomlemma{#1}\innercustomlemma}
  {\endinnercustomlemma}

\newcommand{\Norm}[1]{ \left\|  #1 \right\| }

\newcommand{\floor}[1]{\lfloor #1 \rfloor }

\newcommand{\Be}{\begin{equation}}

\newcommand{\Ee}{\end{equation}}

\newcommand{\Bm}{\begin{multline}}

\newcommand{\Em}{\end{multline}}

\newcommand{\Bea}{\begin{eqnarray}}

\newcommand{\Eea}{\end{eqnarray}}

\newcommand{\Beas}{\begin{eqnarray*}}

\newcommand{\Eeas}{\end{eqnarray*}}

\newcommand{\Benu}{\begin{enumerate}}

\newcommand{\Eenu}{\end{enumerate}}

\newcommand{\Bi}{\begin{itemize}}

\newcommand{\Ei}{\end{itemize}}

\def\intslash{\fint}

\def\intslash{\rlap{\kern  .32em $\mspace {.5mu}\backslash$ }\int}

\def\qsl{{\rlap{\kern  .32em $\mspace {.5mu}\backslash$ }\int_{Q_x}}}

\def\Norm#1{{ \left\|  #1 \right\| }}

\def\floor#1{{\lfloor #1 \rfloor }}

\def\emph#1{{\it #1 }}

\font \roman = cmr10 at 10 true pt

\def\x{{\hbox{\roman x}}}

\def\be#1{\begin{equation}\label{ #1}}

\def\endeq{\end{equation}}

\def\endal{\end{align}}

\def\bas{\begin{align*}}

\def\eas{\end{align*}}

\def\bi{\begin{itemize}}

\def\ei{\end{itemize}}

\def\emph#1{{\it #1}}

\def\textbf#1{{\bf #1}}

\theoremstyle{plain}

   \newtheorem{theorem}{Theorem}[section]

   \newtheorem{proposition}[theorem]{Proposition}

   \newtheorem{lemma}[theorem]{Lemma}

   \newtheorem{corollary}[theorem]{Corollary}

   \newtheorem{theorem*}{Theorem}

\theoremstyle{remark}

   \newtheorem{remark}[theorem]{Remark}

\theoremstyle{definition}

\numberwithin{equation}{section}

\begin{document}
\title[Multipliers associated to convex domains in $\mathbb{R}^2$]{Multiplier transformations associated to convex domains in $\mathbb{R}^2$}

\author[L. Cladek]{Laura Cladek}

\address{L. Cladek, Department of Mathematics\\ University of Wisconsin-Madison\\480 Lincoln Drive, Madison, WI 53706, USA}

\email{cladek@math.wisc.edu}


\begin{abstract}
We consider Fourier multipliers in $\mathbb{R}^2$ of the form $m\circ\rho$ where $\rho$ is the Minkowski functional associated to a convex set in $\mathbb{R}^2$, and prove $L^p$ bounds for the corresponding multiplier operators. It is of interest to consider domains whose boundary is not smooth. Our results depend on a notion of Minkowski dimension introduced in \cite{sz} that measures ``flatness" of the boundary of the domain. Our methods analyze the case of oscillatory multipliers $\frac{e^{i\rho(\xi)}}{(1+|\xi|)^{-a}}$ associated to wave equations, which we use to derive results for more general multiplier transformations.
\end{abstract}

\thanks{The author would like to thank Andreas Seeger for introducing this problem, and for his guidance and many helpful discussions. Research supported in part by NSF Research and Training grant DMS 1147523}
\maketitle


\section{Introduction}
Let $\Omega\subset\mathbb{R}^2$ be a bounded, open convex set such that $0\in\Omega$, and let $\rho$ be its Minkowski functional, given by 
\begin{align*}
\rho(\xi)=\inf\{t>0|\,t^{-1}\xi\in\Omega\}.
\end{align*}
Since $\Omega$ is convex, $\rho:\mathbb{R}^2\to\mathbb{R}^+\cup\{0\}$ is the unique function that is homogeneous of degree one and identically $1$ on $\partial\Omega$. We are interested in multipliers of the form $m\circ\rho$, where $m:\mathbb{R}\to\mathbb{C}$ is a bounded, measurable function. We refer to this class of multipliers as \textit{quasiradial multipliers}. The class of quasiradial multipliers generalizes radial multipliers on $\mathbb{R}^2$, which would correspond to the special case that $\Omega$ is the unit disc and $\rho(\xi)=|\xi|$.
\newline
\indent
As a model case for quasiradial multipliers, one can study the generalized Bochner-Riesz multipliers $(1-\rho(\xi))_+^{\lambda}$ for $\lambda>0$. We define the generalized Bochner-Riesz operators $T_{\lambda}$ for $\lambda>0$ by
\begin{align*}
\mathcal{F}[T_{\lambda}f](\xi)=(1-\rho(\xi))_+^{\lambda}\widehat f(\xi).
\end{align*}
When $\partial\Omega$ is smooth, the problem of $L^p(\mathbb{R}^2)$ boundedness of the generalized Bochner-Riesz operators is well understood. The problem was first completely solved in the special case that $\Omega$ is the unit circle by Fefferman in \cite{feff} and C\'{o}rdoba in \cite{cor}, where it was proven that $T_{\lambda}$ is bounded on $L^p(\mathbb{R}^2)$ if and only if $\lambda>\lambda_0(p):=|\frac{2}{p}-1|-\frac{1}{2}$. This result was then generalized to domains with smooth boundary by Sj\"{o}lin in \cite{sjolin} and H\"{o}rmander in \cite{hor}. 
\newline
\indent
However, for certain convex domains with rough boundary, the critical index $\lambda_0(p)$ can be improved. In \cite{pod2}, Podkorytov considered Bochner-Riesz means associated to polyhedra in $\mathbb{R}^d$ and showed that if $\rho$ is the Minkowski functional of a polyhedron, then $\mathcal{F}^{-1}[(1-\rho(\cdot))_+^{\lambda}]\in L^1$ for $\lambda>0$. In \cite{sz}, Seeger and Ziesler considered Bochner-Riesz means associated to general convex domains in $\mathbb{R}^2$. They obtained a result involving a parameter similar to the Minkowski dimension of $\partial\Omega$, defined by a family of ``balls", or caps, and we state the definition below.
\newline
\indent
For any $p\in\partial\Omega$, we say that a line $\ell,$ is a \textit{supporting line for $\Omega$ at $p$} if $\ell$ contains $p$ and $\Omega$ is contained in the half plane containing the origin with boundary $\ell$. Let $\mathcal{T}(\Omega, p)$ denote the set of supporting lines for $\Omega$ at $p$. Note that if $\partial\Omega$ is $C^1$, then $\mathcal{T}(\Omega, p)$ has exactly one element, the tangent line to $\partial\Omega$ at $p$. For any $p\in\partial\Omega$, $\ell\in \mathcal{T}(\Omega, p)$, and $\delta>0$, define
\begin{align}
B(p, \ell, \delta)=\{x\in\partial\Omega: \text{dist}(x, \ell)<\delta\}.
\end{align}
Let
\begin{align}
\mathcal{B}_{\delta}=\{B(p, \ell, \delta):\,p\in\partial\Omega, \ell\in\mathcal{T}(\Omega, p)\},
\end{align}
and let $N(\Omega, \delta)$ be the minimum number of balls $B\in\mathcal{B}_{\delta}$ needed to cover $\partial\Omega$. Let
\begin{align}\label{kappaomegadef}
\kappa_{\Omega}=\limsup_{\delta\to 0}\frac{\log N(\Omega, \delta)}{\log\delta^{-1}}.
\end{align}
\indent
The parameter $\kappa_{\Omega}$ defined in (\ref{kappaomegadef}) is similar to the upper Minkowski dimension of $\partial\Omega$. It is easy to show that for any convex domain $\Omega$, $0\le\kappa_{\Omega}\le 1/2$ (see \cite{sz} for details). We now mention a few examples of convex domains with particular values of $\kappa_{\Omega}$. Clearly, if $\Omega$ is a polygon, then $\kappa_{\Omega}=0$. For domains with smooth boundary, $\kappa_{\Omega}=1/2$. This can be seen by noting that there is a point where $\partial\Omega$ has nonvanishing curvature, and near this point the contribution to $N(\Omega, \delta)$ is $\approx\delta^{-1/2}$. One may obtain domains with intermediate values of $\kappa_{\Omega}$ by considering Lebesgue functions associated to Cantor sets with appropriate ratios of dissection. For example, let $g: [0, 1]\to [0, 1]$ be the Lebesgue function associated to the standard middle-thirds Cantor set, commonly referred to as the Cantor function. Define $\gamma: [0, 1]\to [-1, -1/2]$ by
\begin{align*}
\gamma(t)=\int_0^t g(s)\,ds-1.
\end{align*}
Let $\Omega$ be the convex domain bounded by the graph of $\gamma$ and the line segments connecting consecutive vertices in the set 
\begin{align*}
\{(1, -1/2); (1, 1); (-1, 1); (-1, -1); (0, -1)\}.
\end{align*}
Then $\kappa_{\Omega}=\frac{\log_3(2)}{(\log_3(2)+1)}$. One may similarly obtain a convex domain $\Omega$ with $\kappa_{\Omega}=\kappa$ for any $\kappa\in (0, 1/2)$ by a similar construction using a Lebesgue function corresponding to a Cantor set of an appropriate ratio of dissection.
\begin{figure}
\begin{tikzpicture}[scale=0.7]
\draw[name=firstline, very thick, blue] (0, 0)--(-2, 4)--(5, 5);
\draw[name=parab, very thick, blue, scale=1, domain=0:5, smooth, variable=\x] plot ({\x}, {0.2*\x*\x});
\draw (0, 0)--(1, 0)--(1, 0.2)--(0, 0.2)--cycle;
\draw[shift={(1, .2)},rotate=20] (0, 0)--(1, 0)--(1, 0.2)--(0, 0.2)--cycle;
\draw[shift={(1+0.87, .2*(3.4969))}, rotate=30] (0, 0)--(1, 0)--(1, 0.2)--(0, 0.2)--cycle; 
\draw[shift={(2.64, .2*(2.63*2.63))}, rotate=40] (0, 0)--(1, 0)--(1, 0.2)--(0, 0.2)--cycle; 
\draw[shift={(3.29, .2*(3.29*3.29))}, rotate=48] (0, 0)--(1.3, 0)--(1.3, 0.2)--(0, 0.2)--cycle; 
\draw[shift={(4.03, .2*(4.03*4.03))}, rotate=56] (0, 0)--(2, 0)--(2, 0.2)--(0, 0.2)--cycle; 
\draw[shift={(-2.05, 3.9))}, rotate=8.13] (0, 0)--(7.2, 0)--(7.2, 0.2)--(0, 0.2)--cycle; 
\draw[shift={(-2.05, 3.9))}, rotate=-63.43] (0, 0)--(4.5, 0)--(4.5, 0.2)--(0, 0.2)--cycle; 
\node at (1, 2.6) {\LARGE $\Omega$};

\end{tikzpicture}
\caption{As an example, here $\Omega$ is a region bounded by two lines and a portion of a parabola. If we assume all rectangles have shorter sidelength equal to $\delta$, then $N(\Omega, \delta)\le 8$. Since a portion of $\partial\Omega$ is smooth with nonvanishing curvature, we have $\kappa_{\Omega}=1/2$.}\label{fig1}
\end{figure}
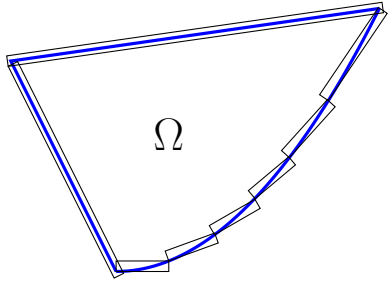
\indent
It was shown in \cite{sz} that $T_{\lambda}$ is bounded on $L^p(\mathbb{R}^2)$ if $\lambda>\kappa_{\Omega}(|\frac{4}{p}-2|-1)$. In this paper we would like to consider more general multiplier transformations. The following subordination formula from \cite{trebels}
\begin{align}\label{subord}
m(\rho(\xi))=\frac{(-1)^{\floor{\lambda}+1}}{\Gamma(\lambda+1)}\int_0^{\infty}s^{\lambda}m^{(\lambda+1)}(s)(1-\frac{\rho(\xi)}{s})_+^{\lambda}\,ds
\end{align}
combined with the result from \cite{sz} mentioned previously immediately gives that $m\circ\rho\in M^p(\mathbb{R}^2)$ if for some $\lambda>\kappa_{\Omega}(|\frac{4}{p}-2|-1)$,
\begin{align*}
\int_0^{\infty}s^{\lambda}|m^{(\lambda+1)}(s)|\,ds<\infty.
\end{align*}
However, this is not satisfactory as can be seen by analyzing the ``localized wave multiplier" $e^{i\rho(\xi)}$. Sharp $L^p$ estimates for this multiplier in the smooth case can be found in \cite{beals}, \cite{miyachi}, \cite{peral} and \cite{sss}. For general convex domains in $\mathbb{R}^2$, we prove the theorem below. First we make a few brief remarks regarding normalization of the domain $\Omega$. Let $\Omega$ be a bounded, open convex set containing the origin, as above. Then $\Omega$ contains some ball centered at the origin and is also contained in some larger ball centered at the origin. Since all results in this paper regarding $L^p$ boundedness of multipliers will be dilation invariant, we will assume without loss of generality that $\Omega$ contains the ball of radius $8$ centered at the origin. Let $M>0$ be an integer such that
\begin{align}\label{ballbound}
\{\xi:\,|\xi|\le 8\}\subset\Omega\subset\overline{\Omega}\subset\{\xi:\,|\xi|<2^M\}.
\end{align}
We will prove
\begin{theorem}\label{symbol1}
Let $\Omega$ be a convex domain satisfying (\ref{ballbound}) and $\rho$ its Minkowski functional. Let $a:\mathbb{R}\to\mathbb{C}$ be a smooth function supported outside \newline$[-2^{-2M}, 2^{-2M}]$ such that $a$ is a symbol of order $-\kappa_{\Omega}-\epsilon$ for some $\epsilon>0$, that is, for every integer $\beta\ge 0$,
\begin{align*}
|D^{\beta}a(\xi)|\lesssim_{\beta}(1+|\xi|)^{-\kappa_{\Omega}-\epsilon-\beta}.
\end{align*}
Then
\begin{align*}
\mathcal{F}^{-1}[a(\rho(\cdot))e^{i\rho(\cdot)}]\in L^1(\mathbb{R}^2),
\end{align*}
where $\Norm{\mathcal{F}^{-1}[a(\rho(\cdot))e^{i\rho(\cdot)}]}_{L^1(\mathbb{R}^2)}$ depends only on $M$, $\epsilon$, and the quantitative estimates for $a$ as a symbol of order $-\kappa_{\Omega}-\epsilon$.
\end{theorem}
The Fourier inversion formula
\begin{align}\label{subord2}
m(\rho(\xi))=\frac{1}{2\pi}\int\widehat m(\tau)e^{i\tau\rho(\xi)}\,d\tau,
\end{align}
which is a more efficient subordination formula than (\ref{subord}), gives the following corollary.

\begin{corollary}\label{symbolcor1}
Let $\Omega$ and $\rho$ be as in the statement of Theorem \ref{symbol1}. For $\epsilon\ge 0$, define
\begin{align*}
\Norm{m}_{B(\kappa_{\Omega}, \epsilon)}:=\int|\widehat{m}(\tau)|(1+|\tau|)^{\kappa_{\Omega}+\epsilon}\,d\tau. 
\end{align*}
If $m$ is a bounded, measurable function supported in $(1/2, 2)$, then
\begin{align*}
\Norm{\mathcal{F}[m\circ\rho]}_{L^1(\mathbb{R}^2)}\lesssim_{\epsilon, M}\Norm{m}_{B(\kappa_{\Omega}, \epsilon)}
\end{align*}
for every $\epsilon>0$.
\end{corollary}

\begin{proof}[Proof that Theorem \ref{symbol1} implies Corollary \ref{symbolcor1}]
Since $m$ is supported in $(1/2, 2)$, there is a smooth cutoff $\chi:\mathbb{R}^2\to\mathbb{R}$ supported compactly away from the origin such that
\begin{align*}
m(\rho(\xi))=\frac{1}{(2\pi)^2}\int\widehat{m}(\tau)\chi(\xi)e^{i\tau\rho(\xi)}\,d\tau.
\end{align*}
We then have
\begin{align*}
\Norm{\mathcal{F}^{-1}[m\circ\rho]}_{L^1(\mathbb{R}^2)}\le\frac{1}{(2\pi)^2}\int|\widehat{m}(\tau)|\Norm{\mathcal{F}^{-1}[\chi(\cdot)e^{i\tau\rho(\cdot)}]}_{L^1(\mathbb{R}^2)}\,d\tau
\\
=\frac{1}{(2\pi)^2}\int|\widehat{m}(\tau)|\Norm{\mathcal{F}^{-1}[\chi(\frac{\cdot}{\tau})e^{i\rho(\cdot)}]}_{L^1(\mathbb{R}^2)}\,d\tau.
\end{align*}
Now, for any $i\ge 0$ and for every $\epsilon>0$,
\begin{align*}
|D_{\xi}^i[\chi(\frac{\xi}{\tau})]|\lesssim_{i, \epsilon, M}(1+|\tau|)^{\kappa_{\Omega}+\epsilon}(1+|\xi|)^{-\kappa_{\Omega}-\epsilon-i},
\end{align*}
and thus Theorem \ref{symbol1} implies that
\begin{align*}
\Norm{\mathcal{F}^{-1}[\chi(\frac{\cdot}{\tau})e^{i\rho(\cdot)}]}_{L^1(\mathbb{R}^2)}\lesssim_{\epsilon, M}(1+|\tau|)^{\kappa_{\Omega}+\epsilon}.
\end{align*}
It follows that
\begin{align*}
\Norm{\mathcal{F}^{-1}[m\circ\rho]}_{L^1(\mathbb{R}^2)}\lesssim_{\epsilon, M}\int|\widehat{m}(\tau)|(1+|\tau|)^{\kappa_{\Omega}+\epsilon}
\end{align*}
for every $\epsilon>0$.
\end{proof}

In the special case that $\kappa_{\Omega}=1/2$, we are able to obtain the following improvement to Theorem \ref{symbol1}.

\begin{theorem}\label{symbol2}
Let $\Omega$ be a convex domain satisfying (\ref{ballbound}) with $\kappa_{\Omega}=1/2$ and $\rho$ its Minkowski functional. Let $a:\mathbb{R}^2\to\mathbb{C}$ be a smooth function supported outside $[-2^{-2M}, 2^{-2M}]$ such that $a$ is a symbol of order $-1/2$, that is, for every integer $\beta\ge 0$,
\begin{align*}
|D^{\beta}a(\xi)|\lesssim_{\beta}(1+|\xi|)^{-1/2-\beta}.
\end{align*}
Then the operator $T$ defined on Schwartz functions $f$ by 
\begin{align*}
\mathcal{F}[Tf](\xi)=a(\rho(\xi))e^{i\rho(\xi)}\mathcal{F}[f](\xi)
\end{align*}
extends to a bounded linear operator from the Hardy space $H^1(\mathbb{R}^2)$ to $L^1(\mathbb{R}^2)$, where the operator norm depends only on $M$ and the quantitative estimates for $a$ as a symbol of order $-1/2$.
\end{theorem}
Using (\ref{subord2}) gives the following corollary.

\begin{corollary}\label{symbolcor2}
Let $\Omega$ and $\rho$ be as in the statement of Theorem \ref{symbol2}. Let $m:\mathbb{R}\to\mathbb{C}$ be a bounded, measurable function supported in $(1/2, 2)$. Then for $1<p<\infty$, the operator $T$ defined on Schwartz functions $f$ by  
\begin{align*}
\mathcal{F}[Tf]=m(\rho(\xi))\mathcal{F}[f]
\end{align*}
extends to a bounded operator on $L^p(\mathbb{R}^2)$, and 
\begin{align*}
\Norm{T}_{H^1(\mathbb{R}^2)\to L^1(\mathbb{R}^2)}\lesssim_{M}\Norm{m}_{B_{1/2, 0}}.
\end{align*}
\end{corollary}

The proof that Theorem \ref{symbol2} implies Corollary \ref{symbolcor2} is similar to the proof that Theorem \ref{symbol1} implies Corollary \ref{symbolcor1}, and is left to the reader.
\newline
\indent
Finally, we would like to remark that while the proof of Theorem \ref{symbol1} draws heavily on ideas from \cite{sz} and \cite{sss}, the proof of Theorem \ref{symbol2} requires the introduction of new techniques.

\subsection*{Generalizations of Theorem \ref{symbol1}}
Theorem \ref{symbol1} applies only to multipliers supported compactly away from the origin. Using Calder\'{o}n-Zygmund theory, we may generalize the result of Theorem \ref{symbol1} to multipliers with non-compact support.
\begin{theorem}\label{multscalethm}
Fix a smooth function $\phi$ supported compactly away from the origin. Let $m$ be a measurable function on $\mathbb{R}$ with $\Norm{m}_{\infty}\le 1$. Let $T$ be the operator defined on Schwartz functions $f$ by \begin{align*}
\mathcal{F}[Tf](\xi)=m(\rho(\xi))\mathcal{F}[f](\xi).
\end{align*}
Then for every $\epsilon>0$ and $1<p<\infty$,
\begin{align*}
\Norm{m\circ\rho}_{M^p}\lesssim_{\epsilon, p}\sup_{t>0}\Norm{\phi(\cdot)m(t\cdot)}_{B_{\kappa_{\Omega}, \epsilon}}.
\end{align*}
\end{theorem}

Theorem \ref{multscalethm} follows immediately from Theorem \ref{symbol1} and the following result from \cite{s}, which we state without proof.

\begin{customprop}{A}[Seeger, \cite{s}]
Suppose that $\sup_{t>0}\Norm{\phi(m(t\cdot))}_{M^p}<\infty$, for some $p\in (1, \infty)$. If  for some $\epsilon>0$, $\sup_{t>0}\Norm{\phi(m(t\cdot))}_{\Lambda_{\epsilon}}<\infty$, then $m\in M_r$, $|1/r-1/2|<|1/p-1/2|.$
\end{customprop}

We will also see in Section \ref{bochnersec} that $L^4(\mathbb{R}^2)$ estimates for a generalized Bochner-Riesz square function leads to a multiplier theorem for quasiradial multipliers in the range $4/3\le p\le 4$. In Section \ref{anasec}, we interpolate this with the result of Theorem \ref{multscalethm} to obtain our final, most general version of Theorem \ref{symbol1}.

\begin{theorem}\label{mainmultthm}
Fix a smooth function $\phi$ supported compactly away from the origin. Let $m$ be a measurable function on $\mathbb{R}$ with $\Norm{m}_{\infty}\le 1$. Let $T$ be the operator defined on Schwartz functions $f$ by \begin{align*}
\mathcal{F}[Tf](\xi)=m(\rho(\xi))\mathcal{F}[f](\xi).
\end{align*}
Let $0\le\theta\le 1$. Then for every $\epsilon>0$ and $\frac{4}{4-\theta}<p<\frac{4}{\theta}$,
\begin{multline*}
\Norm{m\circ\rho}_{M^p}
\\
\lesssim_{\epsilon, p}\sup_{t>0}\bigg(\int|\mathcal{F}_{\mathbb{R}}[\phi(\cdot)m(t\cdot)](\tau)|^{\frac{2}{2-\theta}}(1+|\tau|)^{\frac{2\kappa_{\Omega}+\theta(1-2\kappa_{\Omega})}{2-\theta}+\epsilon}\,d\tau\bigg)^{\frac{2-\theta}{2}}.
\end{multline*}
\end{theorem}

\subsection*{Notation}
We now introduce some notation that will be used throughout the rest of the paper. Given a function $f:X\to\mathbb{R}$ and subsets $A\subsetneq B\subset X$, we will write $A\prec f\prec B$ to indicate that $f$ is identically $1$ on $A$ and supported in $B$. Many of our estimates will have constants that depend on the quantity $M$ associated with $\Omega$ given in (\ref{ballbound}). For the sake of convenience, we will often choose to supress this dependence in our notation. Thus we will use the symbols $\lesssim$ and $\approx$ to denote an inequality where the implied constant possibly depends on $M$.

\section{Preliminaries on convex domains in $\mathbb{R}^2$}\label{prelim}
In this section we state some useful facts about convex domains in $\mathbb{R}^2$. Most of these can be found in \cite{sz}, but we include them here for the sake of completeness. Let $\Omega\subset\mathbb{R}^2$ be a bounded, open convex set containing the origin and satisfying (\ref{ballbound}). The proof of the following lemma is straightforward and uses only elementary facts about convex functions; for more details see \cite{sz}.
\begin{customlemma}{B}[Seeger and Ziesler, \cite{sz}]\label{elemlemma}
$\partial\Omega\cap\{x:\,-1\le x_1\le 1,\,x_2\le 0\}$ can be parametrized by
\begin{align}
t\mapsto (t, \gamma(t)),\qquad{}-1\le t\le 1,
\end{align}
where
\begin{enumerate}
\item
\begin{align}
1<\gamma(t)<2^M,\qquad{} -1\le t\le 1.
\end{align}
\item $\gamma$ is a convex function on $[-1, 1]$, so that the left and right derivatives $\gamma_L^{\prime}$ and $\gamma_R^{\prime}$ exist everywhere in $(-1, 1)$ and
\begin{align}
-2^{M-1}\le\gamma_R^{\prime}(t)\le\gamma_L^{\prime}(t)\le 2^{M-1}
\end{align}
for $t\in [-1, 1]$. The functions $\gamma_L^{\prime}$ and $\gamma_R^{\prime}$ are decreasing functions; $\gamma_L^{\prime}$ and $\gamma_R^{\prime}$ are right continuous in $[-1, 1]$.
\item Let $\ell$ be a supporting line through $\xi\in\partial\Omega$ and let $n$ be an outward normal vector. Then
\begin{align}
|\left<\xi, n\right>|\ge 2^{-M}|\xi|.
\end{align}
\end{enumerate}
\end{customlemma}

\subsection*{Decomposition of $\partial\Omega$} As another preliminary ingredient, we need the decomposition of $\partial\Omega\cap \{x:\,-1\le x_1\le 1,\,x_2<0\}$ introduced in \cite{sz}. This decomposition allows us to write $\partial\Omega$ as a disjoint union of pieces on which $\partial\Omega$ is sufficiently ``flat", where the number of pieces in the decomposition is closely related to the covering numbers $N(\Omega, \delta)$. We inductively define a finite sequence of increasing numbers
\begin{align*}
\mathfrak{A}(\delta)=\{a_0, \ldots, a_Q\}
\end{align*}
as follows. Let $a_0=-1$, and suppose $a_0, \ldots, a_{j-1}$ are already defined. If
\begin{align}\label{case1}
(t-a_{j-1})(\gamma_L^{\prime}(t)-\gamma_R^{\prime}(a_{j-1}))\le\delta\text{ for all }t\in (a_{j-1}, 1])
\end{align}
and $a_{j-1}\le 1-2^{-M}\delta$, then let $a_j=1$. If (\ref{case1}) holds and $a_{j-1}>1-2^{-M}\delta$, then let $a_j=a_{j-1}+2^{-M}\delta$. If (\ref{case1}) does not hold, define
\begin{align*}
a_j=\inf\{t\in (a_{j-1}, 1]:\,(t-a_{j-1})(\gamma_L^{\prime}(t)-\gamma_R^{\prime}(a_{j-1}))>\delta\}.
\end{align*}
Now note that (\ref{case1}) must occur after a finite number of steps, since we have $|\gamma_L^{\prime}|, |\gamma_R^{\prime}|\le 2^{M-1}$, which implies that $|t-s||\gamma_L^{\prime}(t)-\gamma_R^{\prime}(s)|<\delta$ if $|t-s|<\delta 2^{-M}$. Therefore this process must end at some finite stage $j=Q$, and so it gives a sequence $a_0<a_1<\cdots<a_Q$ so that for $j=0, \ldots, Q-1$
\begin{align}\label{left} 
(a_{j+1}-a_j)(\gamma_L^{\prime}(a_{j+1})-\gamma_R^{\prime}(a_j))\le\delta,
\end{align}
and for $0\le j<Q-1$,
\begin{align}\label{right}
(t-a_j)(\gamma_L^{\prime}(t)-\gamma_R^{\prime}(a_j))>\delta\qquad\text{if }t>a_{j+1}.
\end{align}
For a given $\delta>0$, this gives a decomposition of 
\begin{align*}
\partial\Omega\cap\{x:\,-1\le x_1\le 1, x_2<0\}
\end{align*} 
into pieces
\begin{align*}
\bigsqcup_{n=0, 1, \ldots, Q-1}\{x\in\partial\Omega:\, x_1\in [a_n, a_{n+1}]\}.
\end{align*}
The number $Q$ in (\ref{left}) and (\ref{right}) is also denoted by $Q(\Omega, \delta)$. Let $R_{\theta}$ denote rotation by $\theta$ radians. The following lemma relates the numbers $Q(R_{\theta}\Omega, \delta)$ to the covering numbers $N(\Omega, \delta)$.
\begin{customlemma}{C}[Seeger and Ziesler, \cite{sz}]\label{covlemma}
There exists a positive constant $C_M$ so that the following statements hold.
\begin{enumerate}
\item $Q(\Omega, \delta)\le C_M\delta^{-1/2}$.
\item $0\le\kappa_{\Omega}\le 1/2$.
\item For any $\theta$, 
\begin{align*}
Q(R_{\theta}\Omega, \delta)\le C_MN(\Omega, \delta)\log(2+\delta^{-1}).
\end{align*}
\item For $\nu=1, \ldots, 2^{2M}$ let $\theta_{\nu}=\frac{2\pi\nu}{2^{2M}}$. Then
\begin{align*}
C_M^{-1}N(\Omega, \delta)\le\sum_{\nu}Q(R_{\theta_{\nu}}\Omega, \delta)\le C_MN(\Omega, \delta)\log(2+\delta^{-1}).
\end{align*}
\end{enumerate}
\end{customlemma}

We may think of $\mathfrak{A}(\delta)$ as a partition of $[-1, 1]$ into intervals. For the purpose of defining a partition of unity, we wish to refine this partition so that consecutive intervals have comparable length, and we construct such a refinement in the proof of the lemma below. Note the improvement to (\ref{count}) in the special case that $\kappa_{\Omega}=1/2$; this will be used later when we prove Theorem \ref{symbol2}. 
\begin{lemma}\label{endpointlemma}
Suppose that $\Omega$ is a convex domain satisfying (\ref{ballbound}). Let $\delta>0$, and let
\begin{align*}
\mathfrak{A}(\delta)=\{a_0, a_1, \ldots a_Q\}
\end{align*}
be the decomposition of $[-1, 1]$ constructed previously, where $a_0=-1$ and $a_1=1$. There exists a refinement 
\begin{align}
\tilde{\mathfrak{A}}(\delta)=\{b_0, b_1, \ldots b_{\tilde{Q}}\}
\end{align} 
of $\mathfrak{A}(\delta)$ with $b_0=-1$ and $b_{\tilde{Q}}=1$, and satisfying the following properties:
\begin{enumerate}
\item 
\begin{align}\label{count}
\text{card}(\tilde{\mathfrak{A}}(2^{-k}))\lesssim k^2N(\Omega, 2^{-k}).
\end{align}
\item Set $I_j=[b_j, b_{j+1}]$. For every $1\le j\le\tilde{Q}$,
\begin{align}\label{slope}
(\gamma^{\prime}(b_j)-\gamma^{\prime}(b_{j-1}))|I_{j-1}|\le 2^{-k}.
\end{align}
\item For every $1\le j\le \tilde{Q}$,
\begin{align}\label{coomp}
|I_{j-1}|/8\le |I_j|\le 8|I_{j-1}|.
\end{align}
\item
\begin{align}\label{sumconstraint}
\sum_j\delta|I_j|^{-1}\lesssim 1.
\end{align}
\end{enumerate}
In the special case that $\kappa_{\Omega}=1/2$, we also have
\begin{align}\label{newcard}
\text{card}(\tilde{\mathfrak{A}}(\delta))\lesssim\delta^{-\kappa_{\Omega}}.
\end{align} 
\end{lemma}

\begin{proof}[Proof of Lemma \ref{endpointlemma}]
We construct $\tilde{\mathfrak{A}}(\delta)$ as follows. For each $0\le j\le Q-1$, let $\tilde{a}_j$ be the midpoint between $a_j$ and $a_{j+1}$, and consider the set
\begin{align*}
A:=\{a_0, \tilde{a}_0, a_1, \tilde{a}_1, \ldots, \tilde{a}_{Q-1}, {a}_Q\}.
\end{align*}
For $x\in A$, let $x^-:=\max\{y\in A:\,y<x\}$ and $x^+:=\min\{y\in A:\,y>x\}$. For every $x\in A$, we define a set of points $B_x$ as follows. If $x$ satisfies $x^+-x=x-x^{-}$, set $B_x=\{x\}$. If $x$ satisfies $x^+-x>x-x^-$, then iteratively define $B_x$ to be the set of $\lesssim\log(1/\delta)$ many points $B_x=\{y_0, y_1,\ldots, y_N\}$ where $y_0$ is the midpoint between $x$ and $x^+$, and for every $k\ge 0$ set $y_{k+1}$ to be the midpoint between $y_{k}$ and $x$, and stop at the first stage $N$ such that $y_N-x\le x-x^-$. Similarly, if $x$ satisfies $x^+-x<x-x^-$, then iteratively define $B_x$ to be the set of $\lesssim\log(1/\delta)$ many points $B_x=\{y_0, y_1,\ldots, y_N\}$ where $y_0$ is the midpoint between $x$ and $x^-$, and for every $k\ge 0$ set $y_{k+1}$ to be the midpoint between $y_{k}$ and $x$, and stop at the first stage $N$ such that $x-y_N\le x^+-x$. Now let 
\begin{align*}
\tilde{\mathfrak{A}}(\delta)=\bigcup_{x\in A}B_x.
\end{align*}
Clearly, $\tilde{\mathfrak{A}}(\delta)$ satisfies (\ref{slope}), since any refinement of $\mathfrak{A}(\delta)$ automatically satisfies (\ref{slope}). It is also obvious that $\tilde{\mathfrak{A}}(\delta)$ satisfies (\ref{coomp}). Since $\mathfrak{A}(\delta)$ satisfies (\ref{right}), we have
\begin{align*}
\sum_j2^{-k}|I_j|^{-1}\lesssim\sum_j 2^{-k}(a_{j+1}-a_j)^{-1}\lesssim\sum_j(\gamma^{\prime}(a_{j+1})-\gamma^{\prime}(a_j))\lesssim 1,
\end{align*}
so $\tilde{\mathfrak{A}}(\delta)$ satisfies (\ref{sumconstraint}). By Lemma \ref{covlemma}, we have
\begin{align}\label{temp1}
\text{card}(\tilde{\mathfrak{A}}(2^{-k}))=\tilde{Q}+1\lesssim k\cdot\text{card}(\mathfrak{A}(2^{-k}))\lesssim k^2N(\Omega, 2^{-k}).
\end{align}
and so $\tilde{\mathfrak{A}}(\delta)$ satisfies (\ref{count}).
\newline
\indent
In the case that $\kappa_{\Omega}=1/2$, we note that (\ref{right}) implies that for any $L>0$, the number of intervals $[a_j, a_{j+1}]$ such that $(a_{j+1}-a_j)\approx L$ is $\lesssim\min(L\delta^{-1}, L^{-1})$. Thus for any $r>0$ the number of pairs \newline$\big([a_j, a_{j+1}]; [a_{j+1}, a_{j+2}]\big)$ with 
\begin{align*}
\max\bigg(\frac{a_{j+2}-a_{j+1}}{a_{j+1}-a_j}, \frac{a_{j+1}-a_j}{a_{j+2}-a_{j+1}}\bigg)\approx r
\end{align*}
is $\lesssim r^{-1}\delta^{-1/2}$. It follows that the number of points $x\in A$ with 
\begin{align*}
\max\bigg(\frac{x^+-x}{x-x^-}, \frac{x-x^-}{x^+-x}\bigg)\approx r
\end{align*}
is $\lesssim r^{-1}\delta^{-1/2}$. For such points $x$ we have $\text{card}(B_x)\lesssim\log(r)$, and so summing over all dyadic $r=2^k$ we have that
\begin{align*}
\sum_{k\ge 0}k2^{-k}\delta^{-1/2}\lesssim\delta^{-1/2},
\end{align*}
and hence $\tilde{\mathfrak{A}}(\delta)$ satisfies $(\ref{newcard})$.
\end{proof}

\subsection*{Approximating $\Omega$ by convex domains with smooth boundary}
It will be necessary to approximate $\Omega$ by a sequence of convex domains with smooth boundaries. In \cite{sz}, this was done by approximating $\Omega$ by a sequence of convex polygons with sufficiently many vertices and smoothing out the boundary near the vertices. We state the following lemma from \cite{sz} without proof.
\begin{customlemma}{D}[Seeger and Ziesler, \cite{sz}]\label{approxlemma}
Let $\Omega\subset\mathbb{R}^2$ be an open convex domain containing the origin. There is a sequence of convex domains $\{\Omega_n\}$ containing the origin, with Minkowski functionals $\rho_n(\xi)=\inf\{t>0|\,\xi/t\in\Omega_n\}$, so that the following holds:
\begin{enumerate}
\item $\Omega_n\subset\Omega_{n+1}\subset\Omega$ and $\bigcup_n\Omega_n=\Omega$.
\item $\rho_n(\xi)\ge\rho_{n+1}(\xi)\ge\rho(\xi)$ and
\begin{align*}
\frac{\rho_n(\xi)-\rho(\xi)}{\rho(\xi)}\le 2^{-n-1};
\end{align*}
in particular $\lim_{n\to\infty}\rho_n(\xi)=\rho(\xi)$, with uniform convergence on compact sets.
\item $\Omega_n$ has $C^{\infty}$ boundary.
\item If $\delta\ge 2^{-n+2}$ then 
\begin{align*}
N(\Omega_n, 2\delta)\le N(\Omega, \delta).
\end{align*}
\end{enumerate}
\end{customlemma}

\subsection*{Computing $\nabla\rho$} Assuming that $\rho\in C^1(\mathbb{R}^2\setminus\{0\})$, we would like to compute $\nabla\rho(\alpha, \gamma(\alpha))$ for $\alpha\in [-1, 1]$. Since $\nabla\rho$ is homogeneous of degree $0$, this will actually give us $\nabla\rho(\xi)$ for any $\xi$ in a sector of $\mathbb{R}^2\setminus\{0\}$ bounded by rays passing through $(-1, \gamma(-1))$ and $(1, \gamma(1))$. Note that 
\begin{align}\label{gradperp}
\nabla\rho(\alpha, \gamma(\alpha))\cdot (1, \gamma^{\prime}(\alpha))=0,
\end{align}
and thus $\nabla\rho(\alpha, \gamma(\alpha))$ is parallel to $(-\gamma^{\prime}(\alpha), 1)$.
Differentiating the homogeneity relation
\begin{align*}
\rho(t(\alpha, \gamma(\alpha)))=t\rho(\alpha, \gamma(\alpha))
\end{align*}
with respect to $t$ and setting $t=1$ yields
\begin{align}\label{dot1}
(\nabla\rho(\alpha, \gamma(\alpha)))\cdot (\alpha, \gamma(\alpha))=1.
\end{align}
It follows that 
\begin{align}\label{gradlength}
|\nabla\rho(\alpha, \gamma(\alpha))|=\frac{|(-\gamma^{\prime}(\alpha), 1)|}{|\left<(\alpha, \gamma(\alpha)); (-\gamma^{\prime}(\alpha), 1)\right>|}.
\end{align}
Note that (\ref{ballbound}) implies that 
\begin{align}\label{innprod}
|\left<(\alpha, \gamma(\alpha)); (-\gamma^{\prime}(\alpha), 1)\right>|\ge 2^{-4M}.
\end{align}
Together (\ref{gradperp}) and (\ref{gradlength}) imply that
\begin{align}\label{gradequals}
\nabla\rho(\alpha, \gamma(\alpha))=\frac{(\gamma^{\prime}(\alpha), -1)}{\alpha\gamma^{\prime}(\alpha)-\gamma(\alpha)}.
\end{align}
Note that (\ref{ballbound}) and (\ref{gradequals}) implies that 
\begin{align}\label{gradbound}
|\nabla\rho(\alpha, \gamma(\alpha))|\le 2^{5M}.
\end{align}

\section{$L^1$ kernel estimates}\label{l1kerest}
The goal of this section is to prove Theorem \ref{symbol1}. Let $\Omega$, $\rho$ and $a$ be as in the statement of Theorem \ref{symbol1}. Motivated by \cite{sss}, we would like to perform a dyadic decomposition of the multiplier $a(\rho(\xi))e^{i\rho(\xi)}$. Let $\{\theta_k\}_{k\ge0}$ be a smooth dyadic partition of unity of $\mathbb{R}$, so that $\theta_0$ is supported in $[-2^{-3M}, 2^{-3M}]$ and $\theta_k$ is supported in an annulus $|\xi|\approx 2^{k-3M}$ for $k>0$. We write
\begin{align*}
K(x):=\mathcal{F}^{-1}[a(\rho(\cdot))e^{i\rho(\cdot)}](x)=\sum_{k\ge 0}K_k(x),
\end{align*}
where
\begin{align}\label{kkdef}
K_k(x):=\mathcal{F}^{-1}[a(\rho(\cdot))e^{i\rho(\cdot)}\theta_k(\rho(\cdot))](x).
\end{align}
It is easy to see that Theorem \ref{symbol1} is a consequence of the following.

\begin{proposition}\label{mainprop1}
Let $\Omega$, $\rho$ and $a$ be as in the statement of Theorem \ref{symbol1}. Define $K_k$ as in (\ref{kkdef}). Then for $k>0$ and for every $\epsilon>0$,
\begin{align*}
\Norm{K_k}_{L^1(\mathbb{R}^2)}\lesssim_{\epsilon}2^{-k\epsilon/2}.
\end{align*}
\end{proposition}

In order to obtain kernel estimates using techniques similar to those in \cite{sz}, we want to work with domains with smooth boundaries, rather than arbitrary convex domains for which the boundary need only be Lipschitz. Thus we will use Lemma \ref{approxlemma} to reduce Proposition \ref{mainprop1} to the following.
\begin{proposition}\label{mainprop2}
Let $\Omega$, $\rho$ and $a$ be as in the statement of Theorem \ref{symbol1}. Fix an integer $k>0$. Let $\tilde{\Omega}$ be a convex domain with smooth boundary such that
\begin{align*}
\{\xi:\,|\xi|\le 4\}\subset\tilde{\Omega}\subset\overline{\tilde{\Omega}}\subset\{\xi:\,|\xi|<2^{M+1}\},
\end{align*}
and such that 
\begin{align}\label{derp}
N(\tilde{\Omega}, 2^{-k})\le N(\Omega, 2^{-k-1}).
\end{align} 
Let $\tilde{\rho}$ be the Minkowski functional of $\tilde{\Omega}$. Define 
\begin{align*}
\tilde{K}_k(x):=\mathcal{F}^{-1}[a(\tilde{\rho}(\cdot))e^{i\tilde{\rho}(\cdot)}\theta_k(\tilde{\rho}(\cdot))](x).
\end{align*}
Then for every $\epsilon>0$,
\begin{align*}
\Norm{\tilde{K}_k}_{L^1(\mathbb{R}^2)}\lesssim_{\epsilon}2^{-k\epsilon/2}.
\end{align*}
\end{proposition}

\begin{proof}[Proof that Proposition \ref{mainprop2} implies Proposition \ref{mainprop1}]
Let $\{\rho_n\}$ be a sequence of Minkowski functionals approximating $\rho$ as in Lemma \ref{approxlemma}, and for each $n$ set
\begin{align*}
K_{k, n}(x):=\mathcal{F}[a(\rho_n(\cdot))e^{i\rho_n(\cdot)}\theta_k(\rho_n(\cdot))](x).
\end{align*}
Since $\rho_n\to\rho$ uniformly on compact sets, $K_{k, n}(x)\to K_k(x)$ pointwise almost everywhere, and so Fatou's lemma yields
\begin{align*}
\Norm{K_k}_{L^1(\mathbb{R}^2)}\le\liminf_{n\to\infty}\Norm{K_{k, n}}_{L^1(\mathbb{R}^2)}\lesssim_{\epsilon}2^{-k\epsilon/2},
\end{align*}
where in the second to last step we have applied Proposition \ref{mainprop2}.
\end{proof}

Now that we have reduced Proposition \ref{mainprop1} to Proposition \ref{mainprop2} we may now work with distance functions $\tilde{\rho}$ that are smooth away from the origin, and so we may express the kernels in homogeneous coordinates (polar coordinates associated to $\tilde{\Omega}$) and integrate by parts. This  is the general approach used in \cite{sz} to handle the generalized Bochner-Riesz multipliers. We emphasize that we must take care to ensure that our estimates ultimately depend only on the $C^1$ norm of $\partial\tilde{\Omega}$, which is bounded by $2^M$ (and not, for instance, the $C^2$ norm). That this is necessary can be seen in the statements of Theorem \ref{symbol1}, Proposition \ref{mainprop1} and Proposition \ref{mainprop2}, where none of the constants in the estimates to be proven depend on the $C^2$ norm of $\partial\tilde{\Omega}$. However, if we recall the remarks made about notation in the introduction, each of the constants in these estimates implicitly depend on $M$.

\begin{proof}[Proof of Proposition \ref{mainprop2}]
We first note that after employing an appropriate angular partition of unity and using rotational invariance it suffices to consider $\tilde{K}_k$ multiplied by a smooth angular cutoff on the Fourier side. Thus in what follows we will instead let
\begin{align}\label{kktwiddledef}
\tilde{K}_k(x):=\mathcal{F}^{-1}[a(\tilde{\rho}(\cdot))e^{i\tilde{\rho}(\cdot)}\theta_k(\tilde{\rho}(\cdot))\chi(\cdot)](x)
\end{align}
where $\chi(\xi)=\chi_1(\frac{\xi_1}{|\xi|})\chi_2(\tilde{\rho}(\xi))$ for smooth functions $\chi_1, \chi_2:\mathbb{R}\to\mathbb{R}$ so that $[-2^{-2M-1}, 2^{-2M-1}]\prec\chi_1\prec [-2^{-2M}, 2^{-2M}]$, and so that $\chi_2$ is identically $1$ on the support of $a$ and $0$ in a sufficiently small ball centered at the origin. Let $\gamma$ be a parametrization of $\partial\tilde{\Omega}\cap\{x:\,-1\le x_1\le 1,\, x_2\le 0\}$ as in Lemma \ref{elemlemma}. We introduce homogeneous coordinates
\begin{align}\label{homog}
(s, \alpha)\mapsto\xi(s, \alpha)=(s\alpha, s\gamma(\alpha)).
\end{align}
In this coordinate system, $\{(s, \alpha):\,s=1\}\subset\{\xi:\,\rho(\xi)=1\}$. The map (\ref{homog}) has Jacobian
\begin{align*}
\det\bigg(\frac{\partial\xi}{\partial(s, \alpha)}\bigg)=s(\alpha\gamma^{\prime}(\alpha)-\gamma(\alpha)).
\end{align*}
Note that there is a smooth function $\tilde{\chi}_1:\mathbb{R}\to\mathbb{R}$ so that $\chi_1(\frac{\xi_1}{|\xi|})$ in homogeneous coordinates is given by $\tilde{\chi}_1(\alpha)$. Using (\ref{homog}), we thus have
\begin{multline}\label{k1}
\tilde{K}_{k}(x)=\int_{\mathbb{R}^2}e^{i\tilde{\rho}(\xi)}a(\tilde{\rho}(\xi))\theta_k(\tilde{\rho}(\xi))\chi(\xi)e^{ix\cdot\xi}\,d\xi
\\
=\int_0^{\infty}\int e^{is(\alpha x_1+\gamma(\alpha)x_2+1)}a(s)\theta_k(s){\tilde{\chi}_1}(\alpha)s(\alpha\gamma^{\prime}(\alpha)-\gamma(\alpha))\,d\alpha\,ds.
\end{multline}

\subsection*{Kernel estimates far away from the singular set}
Considering the phase $ix\cdot\xi+i\tilde{\rho}(\xi)$ as a function of the variable $\xi$, we see that its gradient vanishes on the singular set $x\in\{-\nabla\tilde{\rho}(\xi):\,\xi\in\mathbb{R}^2\}$. Since $|\nabla\tilde{\rho}|\le 2^{5M}$ as noted in (\ref{gradbound}), we choose to separately estimate the $L^1$ norm of $\tilde{K}_k$ away from a sufficiently large ball (say, of radius $2^{6M}$) centered at the origin. We would expect that after localization on the Fourier side, the multiplier $e^{i\tilde{\rho}(\xi)}$ acts like translation by $\nabla\tilde{\rho}(\xi_0)$ for some $\xi_0$, and hence we might expect any pointwise kernel estimates we obtain off of the ball of radius $2^{6M}$ centered at the origin to be robust under perturbations by $\nabla\tilde{\rho}(\xi_0)$. Thus we will not further decompose the multiplier $\mathcal{F}[\tilde{K}_k]$ when estimating the $L^1$ norm of $\tilde{K}_k$ off of this ball.
\newline
\indent
Throughout the rest of this paper, $\phi_0:\mathbb{R}\to\mathbb{R}$ will be a smooth function satisfying $[-1/2, 1/2]\prec\phi\prec[-1, 1]$. We set $c=c(\Omega, \epsilon)=\frac{1}{2}\max(\kappa_{\Omega}, \epsilon)$. We will show that
\begin{align}\label{ballest}
\int|\tilde{K}_k(x)(1-\phi_0(2^{-6M}|x|))|\,dx\lesssim 2^{-kc}.
\end{align}
To do this we will first prove
\begin{align}\label{ballest1}
\int|\tilde{K}_k(x)(\phi_0(2^{-3k-6M}|x|)-\phi_0(2^{-6M}|x|))|\,dx\lesssim 2^{-kc}
\end{align}
and then prove
\begin{align}\label{ballest2}
\int|\tilde{K}_k(x)(1-\phi_0(2^{-3k-6M}|x|))|\,dx\lesssim 2^{-k}.
\end{align}
Let $\eta:\mathbb{R}\to\mathbb{R}$ be a smooth function satisfying $[-2^{-3M-1}, 2^{-3M-1}]\prec\eta\prec[-2^{-3M}, 2^{-3M}]$. We decompose
\begin{align*}
\tilde{K}_k(x)(\phi_0(2^{-3k-6M}|x|)-\phi_0(2^{-6M}|x|))=\tilde{K}_{k, 1}(x)+\tilde{K}_{k, 2}(x),
\end{align*}
where 
\begin{multline}\label{kk1}
\tilde{K}_{k, 1}(x)=(\phi_0(2^{-3k-6M}|x|)-\phi_0(2^{-6M}|x|))
\\
\times\int_0^{\infty}\int e^{is(\alpha x_1+\gamma(\alpha)x_2+1)}a(s)\eta\bigg(\frac{x_1+x_2\gamma^{\prime}(\alpha)}{|x|}\bigg)\\
\times\theta_k(s){\tilde{\chi}_1}(\alpha)s(\alpha\gamma^{\prime}(\alpha)-\gamma(\alpha))\,d\alpha\,ds
\end{multline}
and
\begin{multline}\label{kk2}
\tilde{K}_{k, 2}(x)=(\phi_0(2^{-3k-6M}|x|)-\phi_0(2^{-6M}|x|))
\\
\times\int_0^{\infty}\int e^{is(\alpha x_1+\gamma(\alpha)x_2+1)}a(s)\bigg(1-\eta\bigg(\frac{x_1+x_2\gamma^{\prime}(\alpha)}{|x|}\bigg)\bigg)
\\
\times\theta_k(s){\tilde{\chi}_1}(\alpha)s(\alpha\gamma^{\prime}(\alpha)-\gamma(\alpha))\,d\alpha\,ds.
\end{multline}
\begin{figure}
\begin{tikzpicture}[scale=0.7]
\filldraw[thick, fill=blue!20!white, draw= black, ] plot [smooth cycle] coordinates {(-3, 2) (-2, -2) (1, -3) (4, 1) };

\draw[thick, ->] (0, 0)--(5, 0) node[anchor=north west] {$\xi_1$};
\draw[thick, ->] (0, 0)--(0, 4) node[anchor=south east] {$\xi_2$};
\draw[thick, dashed, ->] (0, 0)--(1.5, -4.5);
\draw[thick, dashed, ->] (1, -3)--(2.5, -2.1);
\filldraw
(2.7, -4.5) circle (2pt) node[anchor= north west] {\tiny $\nabla\rho(\alpha, \gamma(\alpha))$};
\draw[thick, ->] (2.7, -4.5)--(3.2, -6) node[anchor=south east] {\tiny $u_2$} ;
\draw[thick, ->] (2.7, -4.5)--(4.2, -3.6) node[anchor=south east] {\tiny $u_1$} ;
\filldraw
(1, -3) circle (2pt) node[anchor=south east] {\tiny $(\alpha, \gamma(\alpha))$};
\filldraw 
(0,0) circle (2pt) node[anchor=north east] {\tiny $(0, 0)$};
\node at (-1, 1) {\LARGE $\tilde{\Omega}$};
\end{tikzpicture}
\caption{The coordinate system from (\ref{invcoord}).}\label{fig2}
\end{figure}
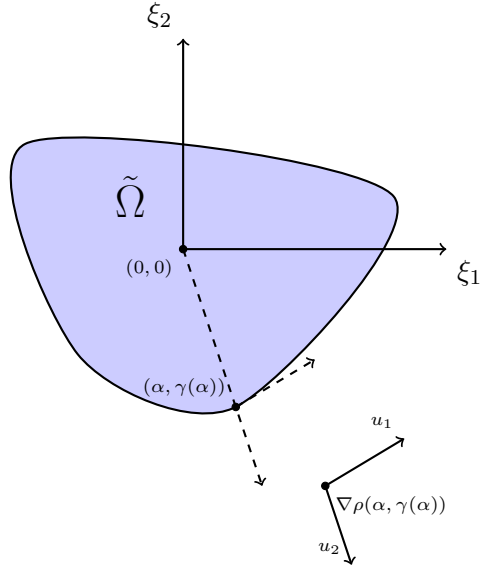

Note that the coordinate system given by the change of coordinates
\begin{align}\label{invcoord}
(x_1, x_2)\mapsto(u_1, u_2):=(x_1+x_2\gamma^{\prime}(\alpha), 1+\alpha x_1+\gamma(\alpha)x_2),
\end{align}
has Jacobian with absolute value $|\alpha\gamma^{\prime}(\alpha)-\gamma(\alpha)|\approx_M 1$. It is also helpful to note that
\begin{align*}
x_1+x_2\gamma^{\prime}(\alpha)=[(x_1, x_2)-\nabla\rho(\alpha,\gamma(\alpha))]\cdot(1, \gamma^{\prime}(\alpha))
\end{align*}
and
\begin{align*}
1+\alpha x_1+\gamma(\alpha)x_2=[(x_1, x_2)-\nabla\rho(\alpha,\gamma(\alpha))]\cdot(\alpha, \gamma(\alpha)),
\end{align*}
and hence our coordinate system is centered at $\nabla\rho(\alpha,\gamma(\alpha))$ with one coordinate direction parallel to $(\alpha, \gamma(\alpha))$ and the other coordinate direction parallel to the tangent vector to $\partial\Omega$ at $(\alpha, \gamma(\alpha))$; see Figure \ref{fig2}. Thus by our choice of the angular cutoff $\chi$ and our choice of $\eta$, it follows that on the support of 
\begin{align*}
(\phi_0(2^{-3k-6M}|x|)-\phi_0(2^{-6M}|x|))\eta\bigg(\frac{x_1+x_2\gamma^{\prime}(\alpha)}{|x|}\bigg)
\end{align*} we have $|x|\approx|1+\alpha x_1+\gamma(\alpha)x_2|$. Similarly, on the support of 
\begin{align*}
(\phi_0(2^{-3k-6M}|x|)-\phi_0(2^{-6M}|x|))\bigg(1-\eta\bigg(\frac{x_1+x_2\gamma^{\prime}(\alpha)}{|x|}\bigg)\bigg)
\end{align*}
we have $|x|\approx |x_1+x_2\gamma^{\prime}(\alpha)|$.
\newline
\indent
Integrating (\ref{kk1}) by parts three times with respect to $s$ and using the above observations yields
\begin{multline}\label{kk12}
\int|\tilde{K}_{k, 1}(x)|\,dx\lesssim2^{-k(\kappa_{\Omega}+\epsilon)}\int\int\tilde{\chi}_1(\alpha)\frac{2^{2k}}{(1+2^k|1+\alpha x_1+\gamma(\alpha)x_2|)^3}\,d\alpha\,dx
\\
\lesssim2^{-k(\kappa_{\Omega}+\epsilon)}\int\frac{2^{2k}}{(1+2^k|x|)^3}\,dx\lesssim 2^{-kc}.
\end{multline}
Integrating by parts (\ref{kk2}) once with respect to $\alpha$, we have
\begin{multline}\label{kk22}
\int|\tilde{K}_{k, 2}(x)|\,dx=(\phi_0(2^{-3k-6M}|x|)-\phi_0(2^{-6M}|x|))
\\\times\int_0^{\infty}\int\partial_{\alpha}g_k(x, \alpha)e^{is(\alpha x_1+\gamma(\alpha)x_2+1)}a(s)\theta_k(s)\,ds\,d\alpha,
\end{multline}
where
\begin{align*}
g_k(x, \alpha)=\frac{\tilde{\chi}_1(\alpha)(\alpha\gamma^{\prime}(\alpha)-\gamma(\alpha))(1-\eta(\frac{x_1+x_2\gamma^{\prime}(\alpha)}{|x|}))}{x_1+x_2\gamma^{\prime}(\alpha)}.
\end{align*}
Integrating by parts (\ref{kk22}) twice with respect to $s$, we have
\begin{multline*}
|\tilde{K}_{k, 2}(x)|\lesssim 2^{-k(\kappa_{\Omega}+\epsilon)}(\phi_0(2^{-3k-6M}|x|)-\phi_0(2^{-6M}|x|))
\\
\times\int|\partial_{\alpha}g_k(x, \alpha)|\frac{2^k}{(1+2^k|\alpha x_1+\gamma(\alpha)x_2+1|)^2}\,d\alpha.
\end{multline*}
Note that on the support of $g_k(x, \alpha)$,
\begin{align}\label{gkest}
|\partial_{\alpha}g_k(x, \alpha)|\lesssim\frac{|\gamma^{\prime\prime}(\alpha)|+1}{|x_1+x_2\gamma^{\prime}(\alpha)|}.
\end{align}
We apply the change of coordinates (\ref{invcoord}). Using (\ref{gkest}), this yields
\begin{multline*}
\int|\tilde{K}_{k, 2}(x)|\,dx
\\
\lesssim 2^{-k(\kappa_{\Omega}+\epsilon)}\int\bigg(\int_{B_{2^{3k+10M}}(0)\setminus B_1(0)}\frac{1}{|u_1|}\frac{2^k}{(1+2^k|u_2|)^2}\,du\bigg)
\\
\times(|\gamma^{\prime\prime}(\alpha)|+1)\tilde{\chi}_1(\alpha)d\alpha
\\
\lesssim 2^{-k(\kappa_{\Omega}+\epsilon)}\int_{B_{2^{3k+10M}}(0)\setminus B_1(0)}\frac{1}{|u_1|}\frac{2^k}{(1+2^k|u_2|)^2}\,du\lesssim k2^{-k(\kappa_{\Omega}+\epsilon)}\lesssim 2^{-kc},
\end{multline*}
which together with (\ref{kk12}) proves (\ref{ballest1}).
\newline
\indent
Now we prove (\ref{ballest2}). We will need the following lemma from \cite{sz}, which we state without proof.
\begin{customlemma}{E}[Seeger and Ziesler, \cite{sz}]\label{szlemma}
Let $h$ be an absolutely continuous function on $[0, \infty)$ and suppose that $\lim_{t\to\infty}h(t)=0$. Suppose that $s\mapsto sh^{\prime}(s)$ defines an $L^1$ function on $[0, \infty)$ and let
\begin{align*}
F(\tau)=\int_0^{\infty}h^{\prime}(s)e^{is\tau}\,ds.
\end{align*}
Suppose that $\mu>0$ and that
\begin{align*}
|F(\tau)|+|F^{\prime}(\tau)|\le B(1+|\tau|)^{-\mu}.
\end{align*}
Let $B(0, R)$ be the ball with radius $R$ and center $0$, and define $\mathcal{A}_l=B(0, 2^{l})\setminus B(0, 2^{l-1})$, for $l>0$, and $\mathcal{A}_0=B(0, 1)$. Then
\begin{align*}
\int_{\mathcal{A}_l}|\mathcal{F}^{-1}[h\circ\rho](x)|\,dx\lesssim_M B[2^{-l(\mu-1)}+l2^{-l}].
\end{align*}
\end{customlemma}
We will apply the lemma with $h(s)=e^{is}a(s)\theta(2^{-k}s)$. Then for every $N>0$,
\begin{align*}
|F(\tau)|+|F^{\prime}(\tau)|\le 2^{k(2-\kappa_{\Omega}-\epsilon)}(1+|\tau|)^{-N},
\end{align*}
and so we conclude that
\begin{align*}
\int_{\mathcal{A}_l}\mathcal{F}^{-1}[h\circ\rho](x)|\,dx\lesssim l2^{k(2-\kappa_{\Omega}-\epsilon)-l}.
\end{align*}
Summing over $l\ge 10k$, we obtain (\ref{ballest2}) and therefore (\ref{ballest}).
\begin{remark}\label{rem}
We note that our proof of (\ref{ballest}) is also valid when $\epsilon=0$ and $\kappa_{\Omega}>0$, which implies $c=\kappa_{\Omega}/2$. We will use this later when we prove an $H^1\to L^1$ endpoint estimate.
\end{remark}
\subsection*{Kernel estimates near the singular set}
It remains to estimate 
\begin{align*}
\int |\tilde{K}_k(x)\phi_0(2^{-6M}|x|)|\,dx.
\end{align*} 
Here we will further decompose the mutiplier $\mathcal{F}[\tilde{K}_k]$ using the decomposition of $\partial\tilde{\Omega}$ from Section \ref{prelim}. Let $\mathfrak{A}(2^{-k})$ be the increasing sequence of numbers associated to $\partial\tilde{\Omega}$ as defined in Section \ref{prelim} with $\delta=2^{-k}$, and let $\tilde{\mathfrak{A}}(2^{-k})$ be the refinement of $\mathfrak{A}(2^{-k})$ as given by Lemma \ref{endpointlemma} and let $\{I_j\}$ be the corresponding partition of $[-1, 1]$ into subintervals. We emphasize that although our collection of intervals $\{I_j\}$ is indexed only by $j$, it implicitly depends on $k$ as well. Now for each such interval $I_j$, let $I_j^{\ast}$ be its $25/24$-dilate (dilated from the center of $I_j$), and let $\{\beta_{I_j}\}$ be a smooth partition of unity subordinate to $\{I_j^{\ast}\}$ such that for each $i\ge 0$,
\begin{align*}
D^i\beta_{I_j}(x)\lesssim |I_j|^{-i}.
\end{align*}
The constant $25/24$ is chosen so that $\{I_j^{\ast}\}$ is an almost-disjoint collection. We decompose 
\begin{align*}
\tilde{K}_k=\sum_j\tilde{K}_{k, j},
\end{align*}
where
\begin{align*}
\tilde{K}_{k, j}(x)=\int_0^{\infty}\int_{I_j^{\ast}}e^{is(\alpha x_1+\gamma(\alpha)x_2+1)}\beta_{I_j}(\alpha)\theta_k(s)a(s)s(\alpha\gamma^{\prime}(\alpha)-\gamma(\alpha))\,d\alpha\,ds,
\end{align*}
that is, $\tilde{K}_{k, j}$ is like $\tilde{K}_k$ with $\beta_{I_j}(\alpha)$ inserted into the integral. We may think of this decomposition on the Fourier side as a decomposition of the multiplier $\mathcal{F}[\tilde{K}_k]$ into smooth functions adapted to sectors bounded by rays originating at the origin and passing through points $(\alpha, \gamma(\alpha))$ where $\alpha\in\tilde{\mathfrak{A}}(2^{-k})$. To estimate $\int|\tilde{K}_{k, j}(x)\phi_0(2^{-6M}|x|)|\,dx$, we will further decompose
\begin{align*}
\tilde{K}_{k, j}(x)\cdot\phi_0(2^{-6M}|x|)=\sum_{n\ge 0}\tilde{K}_{k, j, n}(x),
\end{align*}
where we define $\tilde{K}_{k, j, n}$ as follows. Recall that $\phi_0$ is a smooth function such that $[-1/2, 1/2]\prec\phi_0\prec[-1, 1]$, and let
\begin{align}\label{0def}
\Phi_{k, j, 0}(x, \alpha)=\phi_0(|I_j|2^k(x_1+x_2\gamma^{\prime}(\alpha)))
\end{align}
and for $n>0$ let
\begin{align}\label{ndef}
\Phi_{k, j, n}(x, \alpha)=\phi_0(|I_j|2^{k-n}(x_1+x_2\gamma^{\prime}(\alpha)))-\phi_0(|I_j|2^{k-n+1}(x_1+x_2\gamma^{\prime}(\alpha))).
\end{align}
Set
\begin{align*}
\tilde{K}_{k, j, 0}(x):=\phi_0(2^{-6M}|x|)\int_0^{\infty}\int_{I_j^{\ast}}e^{is(\alpha x_1+\gamma(\alpha)x_2+1)}\beta_{I_j(\alpha)}
\\
\Phi_{k, j, 0}(x, \alpha)
\theta_k(s)a(s)s(\alpha\gamma^{\prime}(\alpha)-\gamma(\alpha))\,d\alpha\,ds
\end{align*}
and for $n>0$ set
\begin{align*}
\tilde{K}_{k, j, n}(x):=\phi_0(2^{-6M}|x|)\int_0^{\infty}\int_{I_j^{\ast}}e^{is(\alpha x_1+\gamma(\alpha)x_2+1)}\beta_{I_j(\alpha)}
\\ \Phi_{k, j, n}(x, \alpha)
\theta_k(s)a(s)
s(\alpha\gamma^{\prime}(\alpha)-\gamma(\alpha))\,d\alpha\,ds,
\end{align*}
that is, $\tilde{K}_{k, j, n}$ is like $\tilde{K}_{k, j}$ with $\Phi_{k, j, n}(x, \alpha)$ inserted into the integral. 
\newline
\indent
To estimate $\int|\tilde{K}_{k, j, 0}(x)|\,dx$, we integrate by parts in $s$ twice to obtain
\begin{multline*}
\int|\tilde{K}_{k, j, 0}(x)|\,dx\lesssim2^{k(1-\kappa_{\Omega}-\epsilon)}\int_{I_j^{\ast}}\int_{|x_1+x_2\gamma^{\prime}(\alpha)|\le |I_j|^{-1}2^{-k}}
2^k
\\
\times(1+2^k|\alpha x_1+\gamma(\alpha)x_2+1|)^{-2}\,dx\,d\alpha.
\end{multline*}
Applying the change of coordinates (\ref{invcoord}) yields
\begin{multline*}
\int|\tilde{K}_{k, j, 0}(x)|\,dx\lesssim2^{k(1-\kappa_{\Omega}-\epsilon)}
\\
\times\int_{I_j^{\ast}}\int_{|u_1|\le |I_j|^{-1}2^{-k}}2^k(1+2^k|u_2|)^{-2}\,du_1\,du_2\,d\alpha
\\
\lesssim 2^{-k(\kappa_{\Omega}+\epsilon)}.
\end{multline*}
By (\ref{count}) and (\ref{derp}), we may sum in $j$ to obtain
\begin{align}\label{0est}
\sum_j\int|\tilde{K}_{k, j, 0}(x)|\,dx\lesssim 2^{-k\epsilon/2}.
\end{align}
Now we estimate $\int |K_{k, j, n}(x)|\,dx$ for $n>0$. Observe that $\tilde{K}_{k, j, n}(x)$ is identically zero when $n\ge k$, so we only need consider the case $n<k$. We integrate by parts once with respect to $\alpha$ and then twice with respect to $s$. Integrating by parts with respect to $\alpha$ yields
\begin{multline*}
\tilde{K}_{k, j, n}(x)=\phi_0(2^{-6M}|x|)\int_0^{\infty}\int_{I_j^{\ast}}\partial_{\alpha}g_{k, j, n}(x, \alpha)e^{is(\alpha x_1+\gamma(\alpha)x_2+1)}
\\a(s)\theta(2^{-k}s)\,ds\,d\alpha,
\end{multline*}
where 
\begin{align*}
g_{k, j, n}(x, \alpha)=\frac{\Phi_{k, j, n}(x, \alpha)\beta_{I_j}(\alpha)(\gamma(\alpha)-\alpha\gamma^{\prime}(\alpha))}{x_1+x_2\gamma^{\prime}(\alpha)}.
\end{align*}
Integrating by parts twice with respect to $s$ yields
\begin{multline*}
|\tilde{K}_{k, j, n}(x)|\lesssim2^{k(-\kappa_{\Omega}-\epsilon)}\phi_0(2^{-6M}|x|)\int_{I_j^{\ast}}|\partial_{\alpha}g_{k, j, n}(x, \alpha)|
\\
\times\frac{2^k}{(1+2^k|\alpha x_1+\gamma(\alpha)x_2+1)|)^{2}}\,d\alpha.
\end{multline*}
Observe that on the support of $\tilde{K}_{k, j, n}(x)$, $|x|\lesssim 1$, so
\begin{multline*}
|\partial_{\alpha}g_{k, j, n}(x, \alpha)|\lesssim\frac{|\gamma^{\prime\prime}(\alpha)|(|I_j|2^{k-n}|x|+1)+|I_j|^{-1}}{|x_1+x_2\gamma^{\prime}(\alpha)|}
\\
\lesssim\frac{|\gamma^{\prime\prime}(\alpha)|(|I_j|2^{k-n}+1)+|I_j|^{-1}}{|x_1+x_2\gamma^{\prime}(\alpha)|}.
\end{multline*}
Thus applying the change of coordinates (\ref{invcoord}), we have
\begin{multline*}
\int|\tilde{K}_{k, j, n}(x)|\,dx\lesssim 2^{k(-\kappa_{\Omega}-\epsilon)}\int_{I_j^{\ast}}(|\gamma^{\prime\prime}(\alpha)|(|I_j|2^{k-n}+1)+|I_j|^{-1})
\\
\times\int_{|u_1|\approx 2^{n-k}|I_j|^{-1}}\frac{1}{|u_1|}\frac{2^k}{(1+2^k|u_2|)^2}\,du\,d\alpha
\\
\lesssim 2^{k(-\kappa_{\Omega}-\epsilon)}\int_{I_j^{\ast}}(|\gamma^{\prime\prime}(\alpha)|(|I_j|2^{k-n}+1)+|I_j|^{-1})\,d\alpha.
\end{multline*}
By (\ref{slope}), if we let $b_j^{\ast}$ and $b_{j+1}^{\ast}$ denote the endpoints of $I_j^{\ast}$, then we have
\begin{align*}
\int_{I_j^{\ast}}|\gamma^{\prime\prime}(\alpha)||I_j|\,d\alpha\lesssim(\gamma^{\prime}(b_{j+1}^{\ast})-\gamma^{\prime}(b_j^{\ast}))|I_j|\lesssim 2^{-k},
\end{align*}
and thus
\begin{align*}
\int|\tilde{K}_{k, j, n}(x)|\,dx\lesssim 2^{k(-\kappa_{\Omega}-\epsilon)}.
\end{align*}
Summing in $j$ and $n$, using (\ref{count}) and (\ref{derp}) and recalling that we only need sum over $n<k$, we obtain
\begin{align}\label{nest}
\sum_j\sum_{n\ge 0}\int|\tilde{K}_{k, j, n}(x)|\,dx\lesssim k2^{-k\epsilon}\lesssim 2^{-k\epsilon/2}.
\end{align}
Combining this with our previous estimates (\ref{0est}) and (\ref{ballest}), we have
\begin{align*}
\int|\tilde{K}_{k}(x)|\,dx\lesssim_{\epsilon} 2^{-k\epsilon/2}, 
\end{align*}
as desired, completing the proof of Proposition \ref{mainprop2} and hence Theorem \ref{symbol1}.
\end{proof}

\section{The $H^1\to L^1$ endpoint estimate: preliminaries and estimate on the exceptional set}\label{end}

In this section, we begin the proof of Theorem \ref{symbol2}. Throughout this section $\kappa_{\Omega}=1/2$. We note that we will often continue to write $\kappa_{\Omega}$ instead of subsituting $1/2$ simply to indicate how certain quantities in our estimates arise. As in the proof of Theorem \ref{symbol1}, the first step is to reduce Theorem \ref{symbol2} to a statement about convex domains with smooth boundary. 

\subsection*{Reduction to the case of smooth boundary}
We invoke Lemma \ref{approxlemma} to show that it suffices to prove Theorem \ref{symbol2} in the special case that $\partial\Omega$ is $C^{\infty}$. For any cube $Q\subset\mathbb{R}^2$, recall that an atom associated to $a_Q$ is a bounded, measurable function supported in $Q$ such that
\begin{align*}
\Norm{a_Q}_{\infty}\le |Q|^{-1},
\\
\int_Qa_Q(x)\,dx=0.
\end{align*}
Let $\phi\ge 0$ be a Schwartz function with compactly supported Fourier transform such that $\Norm{\phi}_{L^1}=1$, and for each $m\ge 0$ let $\phi_m(x)=2^{2m}\phi(2^mx)$. Then there is $N=N(M)>0$ sufficiently large so that
\begin{align*}
\Norm{T(a_Q)}_{L^1}=\lim_{m\to\infty}\Norm{\phi_m\ast(T(a_Q))}_{L^1}=\lim_{m\to\infty}\Norm{{\phi}_m\ast\big(\sum_{k=0}^{2^mN}K_k\ast a_Q\big)}_{L^1},
\end{align*}
where $K_k(x)=\mathcal{F}^{-1}[a(\rho(\cdot))e^{i\rho(\cdot)}\theta_k(\rho(\cdot))](x)$. Let $\{\rho_n\}$ be a sequence of Minkowski functionals approximating $\rho$ as in Lemma \ref{approxlemma}, and let $K_{k, n}(x)=\mathcal{F}^{-1}[a(\rho_n(\cdot))e^{i\rho_n(\cdot)}\theta_k(\rho_n(\cdot))](x)$. Now, assuming that Theorem \ref{symbol2} holds in the special case that $\partial\Omega$ is smooth, for each $m$ we have
\begin{multline*}
\Norm{{\phi}_m\ast\big(\sum_{k=0}^{2^mN}K_k\ast a_Q\big)}_{L^1}\lesssim\liminf_{n\to\infty}\Norm{\phi_m\ast\big(\sum_{k=0}^{2^mN}K_{k, n}\ast a_Q\big)}_{L^1}
\\
\lesssim\liminf_{n\to\infty}\Norm{\sum_{k=0}^{\infty}K_{k, n}\ast a_Q}_{L^1}\lesssim 1,
\end{multline*}
where in the first step above we have used the fact that $\rho_n\to\rho$ uniformly on compact sets. Thus we have shown it suffices to prove Theorem \ref{symbol2} in the special case that $\partial\Omega$ is $C^{\infty}$.

\subsection*{Reduction to the case of cubes with small sidelength}
We assume $\partial\Omega$ is $C^{\infty}$. We need to prove that for any atom $a_Q$,
\begin{align}\label{atomest}
\Norm{T(a_Q)}_{L^1(\mathbb{R}^2)}\le C,
\end{align}
where $C$ is a constant independent of the choice of $Q$ or $a_Q$. 
\newline
\indent
First suppose $Q$ has sidelength $\ge 1$. Let $K(x)=\mathcal{F}^{-1}[a(\rho(\cdot))e^{i\rho(\cdot)}](x)$. Recall that $\phi_0$ is a smooth function such that $[-1/2, 1/2]\prec\phi_0\prec[-1, 1]$. Let $\phi(x)=\phi_0(2^{-6M}|x|)$. Then $(K\phi)\ast a_Q$ is supported in $2^{6M+1}Q$, where the dilation is taken from the center of $Q$. Since $\widehat K\in L^{\infty}$, $\Norm{(K\phi)\ast a_Q}_{2}\lesssim\Norm{a_Q}_2$. By Cauchy-Schwarz,
\begin{align}\label{l2bound}
\Norm{(K\phi)\ast a_Q}_{L^1}\lesssim |Q|^{1/2}\Norm{(K\phi)\ast a_Q}_{L^2}\lesssim |Q|^{1/2}\Norm{a_Q}_{L^2}\lesssim 1.
\end{align}
As stated in Remark \ref{rem}, we have already shown in Section \ref{l1kerest} that
\begin{align*}
\Norm{(K(1-\phi))\ast a_Q}_{L^1}\lesssim 1,
\end{align*}
which proves (\ref{atomest}) if the sidelength of $Q$ is $\ge 1$. 
\newline
\indent
Thus we have reduced Theorem \ref{symbol2} to the following proposition.
\begin{proposition}\label{smoothcube}
Let $\Omega$ be a convex domain with smooth boundary satisfying (\ref{ballbound}), and let $\rho$ be its Minkowski functional. Let $a$ and $T$ be as in the statement of Theorem \ref{symbol2}. Then for every cube $Q$ of sidelength $\le 1$ and for every atom $a_Q$ associated to $Q$, we have
\begin{align*}
\Norm{T(a_Q)}_{L^1(\mathbb{R}^2)}\le C,
\end{align*}
where the constant $C$ depends only on $M$ and the quantitative estimates for $a$ as a symbol of order $-1/2$.
\end{proposition}
We now make the same observation made at the beginning of the proof of Proposition \ref{mainprop2} and note that it is enough to prove Proposition \ref{smoothcube} with the kernel $K$ of the operator $T$ redefined as
\begin{align}\label{newKdef}
K(x):=\mathcal{F}^{-1}[a(\rho(\cdot))e^{i\rho(\cdot)}\theta_k(\rho(\cdot))\chi(\cdot)](x),
\end{align} 
where $\chi$ is the same smooth angular cutoff as in (\ref{kktwiddledef}). Thus in what follows we will take (\ref{newKdef}) to be our definition of $K$.

\subsection*{Estimate on the exceptional set}
In what follows we assume that $Q$ is a cube of sidelength $2^{-l}$ for some $l\ge 0$, and $a_Q$ an atom associated to $Q$. To prove Proposition \ref{smoothcube}, we will define an exceptional set of sufficiently small measure off of which $T(a_Q)$ decays.
Let $\Sigma_{\rho}$ be the smooth closed curve given by 
\begin{align*}
\Sigma_{\rho}:=\{\xi: \xi=-\nabla\rho(\xi^{\prime})\text{ for some }\xi^{\prime}\in\mathbb{R}^2\}.
\end{align*} 
Since $\nabla\rho$ is homogeneous of degree $0$, this indeed corresponds to a smooth closed curve. As noted previously, the gradient of the phase $ix\cdot\xi+i\rho(\xi)$ vanishes on the singular set $\Sigma_{\rho}$. We would like to associate to $Q$ an exceptional set $\mathcal{N}_Q$. A natural choice for $\mathcal{N}_Q$ might be 
\begin{align*}
\{x\in\mathbb{R}^2|\,|x-\Sigma_{\rho}|\le C2^{-l}\}
\end{align*}
for some choice of constant $C$. However, for technical reasons we will choose $\mathcal{N}_Q$ to be a slightly larger set. Let $\{I_j\}$ be the partition of $[-1, 1]$ into subintervals corresponding to the subset $\mathfrak{A}(2^{-l})$ of $[-1, 1]$, as given by Lemma \ref{endpointlemma}. We emphasize that although the collection of intervals $\{I_j\}$ is indexed only by $j$, it implicitly depends on $l$ as well. (Recall that $Q$ has sidelength $2^{-l}$.) For each $j$, choose some $\alpha_j\in I_j$. Define
\begin{align*}
E_{\alpha_j}:=\{x:\,|\alpha_j x_1+\gamma(\alpha_j)x_2+1|\le 2^{-l+15M}, 
\\
|x_1+x_2\gamma^{\prime}(\alpha_j)|\le 2^{-l+15M}|I_j|^{-1}\},
\end{align*}
and define
\begin{align*}
\mathcal{N}_Q:=\bigcup_j E_{\alpha_j}.
\end{align*} 
Then by (\ref{sumconstraint}),
\begin{align*}
|\mathcal{N}_Q|\lesssim \sum_j2^{-2l}|I_j|^{-1}\lesssim 2^{-l}.
\end{align*}
We follow \cite{sss} to estimate $T(a_Q)$ on $\mathcal{N}_Q$. 
By the Hardy-Littlewood-Sobolev inequality,
\begin{align*}
\Norm{(I-\Delta)^{-1/4}a_Q}_2\lesssim\Norm{a_Q}_{4/3}.
\end{align*}
Since $a$ is a symbol of order $-1/2$ and $\rho$ is homogeneous of degree one, the operator $T(I-\Delta)^{-1/4}$ is bounded on $L^2$, and so after using H\"{o}lder's inequality twice we have
\begin{align*}
\Norm{T(a_Q)}_{L^1(\mathcal{N}_Q)}\lesssim 2^{-l/2}\Norm{T(a_Q)}_2\lesssim 2^{-l/2}\Norm{(I-\Delta)^{-1/4}a_Q}_2
\\
\lesssim 2^{-l/2}\Norm{a_Q}_{4/3}\lesssim 1.
\end{align*}
Thus to prove Proposition \ref{smoothcube}, It remains to show 
\begin{align}\label{offexcset}
\Norm{T(a_Q)}_{L^1(\mathbb{R}^2\setminus\mathcal{N}_Q)}\lesssim 1.
\end{align}
As noted in Remark \ref{rem}, we have already shown that 
\begin{align*}
\int|K(x)(1-\phi_0(2^{-6M}|x|))|\,dx\lesssim 1.
\end{align*}
Thus if we let $S$ denote the operator with kernel $K(x)(\phi_0(2^{-6M}|x|)$, (\ref{offexcset}) reduces to proving
\begin{align}\label{offexcset2}
\Norm{S(a_Q)}_{L^1(\mathbb{R}^2\setminus\mathcal{N}_Q)}\lesssim 1.
\end{align}
We now proceed to decompose $S$ as a sum of operators, some of which map $a_Q$ to a function supported inside the exceptional set $\mathcal{N}_Q$; these operators will not contribute to the left hand side of (\ref{offexcset2}). Let $S_k$ denote the operator with kernel $K_k(x)\phi_0(2^{-6M}|x|)$, where 
\begin{align*}
K_k(x)=\mathcal{F}^{-1}[a(\rho(\cdot))e^{i\rho(\cdot)}\theta_k(\rho(\cdot))\chi(\cdot)](x).
\end{align*}
As before, we let $\{I_j\}$ be the collection of intervals corresponding to the partition of $[-1, 1]$ given by $\tilde{\mathfrak{A}}(2^{-l})$, as defined in Section \ref{prelim}. 
\newline
\indent
For each $j$, define
\begin{align*}
\Phi_{l, j, 0}(x, \alpha)=\phi_0(|I_j|2^l(x_1+x_2\gamma^{\prime}(\alpha))).
\end{align*}
For each $j, k$ and for each $n>0$, define
\begin{align*}
\Phi_{k, j, n}(x, \alpha)=\phi_0(|I_j|2^{k-n}(x_1+x_2\gamma^{\prime}(\alpha)))-\phi_0(|I_j|2^{k-n+1}(x_1+x_2\gamma^{\prime}(\alpha))).
\end{align*}
For each $k, j, n\ge 0$, we consider the operators $S_{l, k, j, n}$, $\tilde{S}_{l, k, j}$ and $S^{\prime}_{l, k, j}$ with kernels $L_{l, k, j, n}$, $\tilde{L}_{l, k, j}$ and $L^{\prime}_{l, k, j}$, respectively, given by
\begin{multline}\label{L1}
L_{l, k, j, n}:=\phi_0(2^{-6M}|x|)\int\int e^{is(\alpha x_1+\gamma(\alpha)x_2+1)}\beta_{I_j}(\alpha)
\\
\times\Phi_{k, j, n}(x, \alpha)\theta_k(s)a(s)s(\alpha\gamma^{\prime}(\alpha)-\gamma(\alpha))\chi(\alpha)\,d\alpha\,ds,
\end{multline}  
\begin{multline}\label{L2}
\tilde{L}_{l, k, j}:=\phi_0(2^{-6M}|x|)\int\int e^{is(\alpha x_1+\gamma(\alpha)x_2+1)}\beta_{I_j}(\alpha)
\\
\times\Phi_{l, j, 0}(x, \alpha)(1-\phi_0(2^{l}(\alpha x_1+\gamma(\alpha)x_2+1)))
\\
\times\theta_k(s)a(s)s(\alpha\gamma^{\prime}(\alpha)-\gamma(\alpha))\chi(\alpha)\,d\alpha\,ds
\end{multline}
and
\begin{multline}\label{L3}
{L}^{\prime}_{l, k, j}:=\phi_0(2^{-6M}|x|)\int\int e^{is(\alpha x_1+\gamma(\alpha)x_2+1)}\beta_{I_j}(\alpha)
\\
\times\Phi_{l, j, 0}(x, \alpha)\phi_0(2^{l}(\alpha x_1+\gamma(\alpha)x_2+1))
\\
\times\theta_k(s)a(s)s(\alpha\gamma^{\prime}(\alpha)-\gamma(\alpha))\chi(\alpha)\,d\alpha\,ds.
\end{multline}
Note that $L_{l, k, j, n}(x)$ is like $K_k(x)\phi_0(2^{-6M}|x|)$ with 
\begin{align*}
\beta_{I_j}(\alpha)\cdot\Phi_{k, j, n}(x, \alpha)
\end{align*}
inserted into the integral, $\tilde{L}_{l, k, j}(x)$ is like $K_k(x)\phi_0(2^{-6M}|x|)$ with 
\begin{align*}
\beta_{I_j}(\alpha)\cdot\Phi_{l, j, 0}\cdot(1-\phi_0(2^l(\alpha x_1+\gamma(\alpha)x_2+1)))
\end{align*}
inserted into the integral, and ${L}^{\prime}_{l, k, j}(x)$ is like $K_k(x)\phi_0(2^{-6M}|x|)$ with 
\begin{align*}
\beta_{I_j}(\alpha)\cdot\Phi_{l, j, 0}\cdot\phi_0(2^l(\alpha x_1+\gamma(\alpha)x_2+1))
\end{align*}
inserted into the integral. These kernels are most easily visualized using the coordinate system of (\ref{invcoord}); see Figure \ref{fig3}.
\begin{figure}\label{fig4}

\begin{tikzpicture}[scale=0.5]
\hspace*{-3cm}
\filldraw[thick, fill=blue!20!white, draw= black, ] plot [smooth cycle] coordinates {(-3, 2) (-2, -2) (1, -3) (4, 1) };
\draw[thick, ->] (0, 0)--(5, 0) node[anchor=north west] {$\xi_1$};
\draw[thick, ->] (0, 0)--(0, 5) node[anchor=south east] {$\xi_2$};
\draw[thick, dashed, ->] (0, 0)--(1.5, -4.5);
\draw[thick, dashed, ->] (1, -3)--(2.5, -2.1);
\draw[thick, shift={(.6, -3.3)},rotate=28] (0, 0)--(1.7, 0)--(1.7, 0.4)--(0, 0.4)--cycle;
\filldraw
(.52, -3.1) circle (2pt) node[anchor= north east] {\tiny $(b_j, \gamma(b_j))$};
\filldraw
(1.95, -2.25) circle (2pt) node[anchor= north west] {\tiny $(b_{j+1}, \gamma(b_{j+1}))$};
\filldraw
(1, -3) circle (2pt) node[anchor=north west] {\tiny $(\alpha, \gamma(\alpha))$};
\filldraw 
(0,0) circle (2pt) node[anchor=north east] {\tiny $(0, 0)$};
\node at (-1, 1) {\LARGE $\Omega$};

\hspace{+6cm}

\draw[thick, ->] (0, 0)--(5, 0) node[anchor=south west] {$x_1$};
\draw[thick, ->] (0, 0)--(0, 5) node[anchor=south east] {$x_2$};
\draw[thick, ->] (0, 0)--(0, -5); 
\draw[thick, ->] (0, 0)--(-5, 0);
\draw [name=A, thick, dashed, rotate=31, shift={(4.05, -6.5)}] (0, -1)--(0, 10);
\draw [name=B, thick, dashed, rotate=31, shift={(.2, -6.5)}] (0, .8)--(0, 11.6);
\filldraw[fill=blue!20!white, scale=0.5, thick, shift={(4.1, -5.8)},rotate=20] (-.1, 0)--(7.7, 0)--(7.5, .8)--(-.31, .8)--cycle;
\filldraw
(3.4, -2.25) circle (2pt) node[anchor= north west] {\tiny $\nabla\rho(\alpha, \gamma(\alpha))$};
\draw[thick, ->] (3.4, -2.25)--(3.9, -3.75) node[anchor=south east] {\tiny $u_2$} ;
\draw[thick, ->] (3.4, -2.25)--(4.9, -1.35) node[anchor=south east] {\tiny$u_1$} ;

2.25 
\end{tikzpicture}

\caption{The domain $\Omega$ is depicted on the left, where for a fixed $j$ a point $(\alpha, \gamma(\alpha))$ is chosen so that $\alpha\in I_j$. On the right, up to dilation by a constant, the shaded parallelogram represents the support of $\Phi_{l, j, 0}(x, \alpha)\cdot\phi_0(2^l(\alpha x_1+\gamma(\alpha)x_2+1))$, and up to dilation by a constant the region between the two dashed lines represents the support of $\Phi_{l, j, 0}(x, \alpha)\cdot(1-\phi_0(2^l(\alpha x_1+\gamma(\alpha)x_2+1))$. The region outside the two dashed lines represents the support of $\sum_{n:\, n>k-l}\Phi_{k, j, n}(x, \alpha)$. Note that the long side of the shaded parallelogram is orthogonal to $u_2$, and the dashed lines are orthogonal to $u_1$. The short side of the parallelogram has length $\approx 2^{-l}$, and the long side has length $\approx 2^{-l}|I_j|^{-1}$. }\label{fig3}
\end{figure}
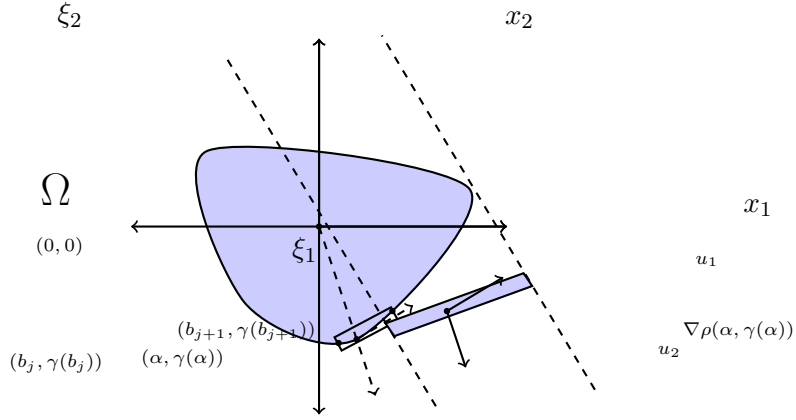

We can write
\begin{align}\label{rep}
S=\sum_{k:\,k<l}S_k+\sum_{k:\,k\ge l}\sum_{n:\,n>k-l}\sum_jS_{l, k, j, n}+\sum_{k:\,k\ge l}\sum_j(\tilde{S}_{l, k, j}+S^{\prime}_{l, k, j}).
\end{align}
If we fix some $k\ge l$ and freeze all sums in $k$ in (\ref{rep}), then we may interpret (\ref{rep}) as follows. The term $\sum_{n:\,n>k-l}\sum_jS_{l, k, j, n}$ may be thought of as the portion of the kernel of $S_k$ supported away in the $u_1$-direction from the exceptional set $\mathcal{N}_Q$, with the distance from $\mathcal{N}_Q$ increasing as $n$ increases. The term $\sum_j\tilde{S}_{l, k, j}$ may be thought of as the portion of the kernel of $S_k$ supported away in the $u_2$-direction from $\mathcal{N}_Q$. We will see that the kernel of the term $\sum_jS^{\prime}_{l, k, j}$ is supported in $\mathcal{N}_Q$. We prove the following lemma.
\begin{lemma}\label{supplemma}
The support of 
\begin{align*}
\sum_{k:\,k\ge l}\sum_jS^{\prime}_{l, k, j}a_Q
\end{align*}
is contained in $\mathcal{N}_Q$.
\end{lemma}

\begin{proof}[Proof of Lemma \ref{supplemma}]
Since $a_Q$ is supported in a cube of sidelength $2^{-l}$, it suffices to show that the kernel of $\sum_{k:\,k\ge l}\sum_jS^{\prime}_{l, k, j}$ is supported in
\begin{align*}
\tilde{\mathcal{N}}_Q:=\bigcup_j\tilde{E}_{\alpha_j}.
\end{align*}
where
\begin{align*}
\tilde{E}_{\alpha_j}:=\{x:\,|\alpha_j x_1+\gamma(\alpha_j)x_2+1|\le 2^{-l+14M}, 
\\
|x_1+x_2\gamma^{\prime}(\alpha_j)|\le 2^{-l+14M}|I_j|^{-1}\}.
\end{align*}
Observe that if we set 
\begin{align*}
c_{\alpha}=\gamma^{\prime}(\alpha)(\alpha\gamma^{\prime}(\alpha)-\gamma(\alpha))^{-1}
\end{align*}
and 
\begin{align*}
d_{\alpha}=-(\alpha\gamma^{\prime}(\alpha)-\gamma(\alpha))^{-1},
\end{align*}
then
\begin{align*}
\alpha x_1+\gamma(\alpha)x_2+1=(\alpha, \gamma(\alpha))\cdot (x_1+c_{\alpha}, x_2+d_{\alpha}),
\end{align*}
and moreover 
\begin{align*}
(c_{\alpha}, d_{\alpha})\cdot (1, \gamma^{\prime}(\alpha))=0.
\end{align*}
In fact, (\ref{gradequals}) states that $(c_{\alpha}, d_{\alpha})=\nabla\rho(\alpha, \gamma(\alpha))$. Now, for any $\alpha, \alpha^{\prime}\in I_j^{\ast}$, (\ref{dot1}) implies that we have
\begin{multline}\label{mess1}
(\alpha, \gamma(\alpha))\cdot(c_{\alpha^{\prime}}-c_{\alpha}, d_{\alpha^{\prime}}-d_{\alpha})=(c_{\alpha^{\prime}}, d_{\alpha^{\prime}})\cdot(\alpha, \gamma(\alpha))-1
\\
=\frac{(\gamma^{\prime}(\alpha^{\prime}), -1)\cdot(\alpha, \gamma(\alpha))}{(\gamma^{\prime}(\alpha^{\prime}), -1)\cdot(\alpha^{\prime}, \gamma(\alpha^{\prime}))}-1.
\end{multline}
By (\ref{left}), we have that 
\begin{align}\label{mess2}
|(\gamma^{\prime}(\alpha^{\prime}), -1)\cdot(\alpha^{\prime}-\alpha, \gamma(\alpha^{\prime})-\gamma(\alpha))|\le 2^{-l+4}.
\end{align}
Indeed, (\ref{mess2}) is equivalent to the statement that $(\alpha, \gamma(\alpha))$ is contained in a rectangle of width $\le 2^{-l+4}$ containing $(\alpha^{\prime}, \gamma(\alpha^{\prime}))$ with short side parallel to the normal to $\partial\Omega$ at $(\alpha^{\prime}, \gamma(\alpha^{\prime}))$. That is, $(\alpha, \gamma(\alpha))$ and $(\alpha^{\prime}, \gamma(\alpha^{\prime}))$ are contained in a single ``Minkowski cap" of width $\delta\le 2^{-l+4}$.
\newline
\indent
As mentioned in (\ref{innprod}), $|(\gamma^{\prime}(\alpha^{\prime}), -1)\cdot(\alpha^{\prime}, \gamma(\alpha^{\prime}))|\ge 2^{-4M}$, and so it follows from (\ref{mess1}) and (\ref{mess2}) that
\begin{multline}\label{mess3}
|(\alpha, \gamma(\alpha))\cdot(c_{\alpha^{\prime}}-c_{\alpha}, d_{\alpha^{\prime}}-d_{\alpha})|
\le\frac{(\gamma^{\prime}(\alpha^{\prime}), -1)\cdot(\alpha^{\prime}-\alpha, \gamma(\alpha^{\prime})-\gamma(\alpha))}{(\gamma^{\prime}(\alpha^{\prime}), -1)\cdot(\alpha^{\prime}, \gamma(\alpha^{\prime}))}
\\
\le 2^{-l+5M}.
\end{multline}
We also note that for any $\alpha, \alpha^{\prime}\in I_j^{\ast}$, 
\begin{multline}\label{mess4}
|(c_{\alpha^{\prime}}-c_{\alpha}, d_{\alpha^{\prime}}-d_{\alpha})|\le 2^{10M} \max(|\gamma^{\prime}(\alpha)-\gamma^{\prime}(\alpha^{\prime})|, |\gamma(\alpha)-\gamma(\alpha^{\prime})|)
\\
\le 2^{10M}\max(2^{-l}|I_j|^{-1}, 2^{-l})\le 2^{-l+10M}|I_j|^{-1},
\end{multline}
where in the second step we have used (\ref{left}). It follows from (\ref{mess3}) and (\ref{mess4}) that for any $\alpha, \alpha^{\prime}\in I_j^{\ast}$,  
\begin{multline}\label{mess5}
\text{supp}\bigg(\phi_0(|I_j|2^l(x_1+x_2\gamma^{\prime}(\alpha^{\prime})))\phi_0(2^l(\alpha^{\prime}x_1+\gamma(\alpha^{\prime})x_2+1))\bigg)
\\\subset\{x:\,(x+(c_{\alpha}, d_{\alpha}))\cdot(1, \gamma^{\prime}(\alpha^{\prime}))\le 2^{-l+12M}|I_j|^{-1},
\\
(x+(c_{\alpha}, d_{\alpha}))\cdot(\alpha^{\prime}, \gamma(\alpha^{\prime}))\le 2^{-l+12M}\}.
\end{multline}
Next, we note that (\ref{ballbound}) implies that for any $\alpha, \alpha^{\prime}\in I_j^{\ast}$, the angle between $(\alpha, \gamma(\alpha))$ and $(\alpha^{\prime}, \gamma(\alpha^{\prime}))$ is $\le |I_j|$, and this combined with (\ref{mess5}) implies that for any $\alpha, \alpha^{\prime}\in I_j^{\ast}$,
\begin{multline}\label{mess6}
\text{supp}\bigg(\phi_0(|I_j|2^l(x_1+x_2\gamma^{\prime}(\alpha^{\prime})))\phi_0(2^l(\alpha^{\prime}x_1+\gamma(\alpha^{\prime})x_2+1))\bigg)
\\
\subset \tilde{E}_{\alpha}:=\{x: |\alpha x_1+\gamma(\alpha)x_2+1|\le 2^{-l+14M},\,|x_1+x_2\gamma^{\prime}(\alpha)|\le 2^{-l+14M}|I_j|^{-1}\},
\end{multline}
and taking $\alpha=\alpha_j$ completes the proof.
\end{proof}
We have thus reduced Proposition \ref{smoothcube}, and hence also Theorem \ref{symbol2}, to the following proposition.

\begin{proposition}\label{offprop}
Let $\tilde{S}_{l, k, j}$, $S_{l, k, j, n}$ and $S_k$ be as defined previously. Then
\begin{align}\label{kglest}
\Norm{\big(\sum_{k: k\ge l}\sum_j(\tilde{S}_{l, k, j}+\sum_{n: n>{k-l}}S_{l, k, j, n})\big)(a_Q)}_{L^1(\mathbb{R}^2)}\lesssim 1
\end{align}
and
\begin{align}\label{kllest}
\Norm{\sum_{k: k<l}S_{k}(a_Q)}_{L^1(\mathbb{R}^2)}\lesssim 1.
\end{align}
\end{proposition}

\section{The $H^1\to L^1$ endpoint estimate: estimate off the exceptional set}
As in the previous section, throughout this section $\kappa_{\Omega}=1/2$. We again note that we will often continue to write $\kappa_{\Omega}$ instead of subsituting $1/2$ simply to indicate how certain quantities in our estimates arise. We have shown that to prove that the operator $S$ maps $a_Q$ into $L^1$, we may ignore the term $\sum_{k: k\ge l}\sum_jS^{\prime}_{l, k, j}$ in (\ref{rep}). All other terms in (\ref{rep}) map $a_Q$ to a function that is supported off the exceptional set. In summary, we have shown that Theorem \ref{symbol2} reduces to proving Proposition \ref{offprop}, and so this section will be devoted to proving Proposition \ref{offprop}.
\subsection*{The case $k\ge l$}
To prove (\ref{kglest}), we will first prove the following lemma.
\begin{lemma}\label{shortlemma}
Let $\tilde{L}_{l, k, j}$ be as defined previously. Then
\begin{align}\label{tildaest}
\sum_{k\ge l}\sum_j\int|\tilde{L}_{l, k, j}(x)|\,dx\lesssim 1.
\end{align}
\end{lemma}
\begin{proof}[Proof of Lemma \ref{shortlemma}]
Integrating by parts (\ref{L2}) three times with respect to $s$ yields
\begin{multline*}
\int|\tilde{L}_{l, k, j}(x)|\,dx\lesssim 
\\2^{k(1-\kappa_{\Omega})}\int_{{I_j}^{\ast}}\int_{\substack{|x_1+x_2\gamma^{\prime}(\alpha)|\le |I_j|^{-1}2^{-l}\\|\alpha x_1+\gamma(\alpha) x_2+1|\gtrsim 2^{-l}}}\frac{2^k}{(1+2^k|\alpha x_1+\gamma(\alpha) x_2+1|)^{3}}\,dx\,d\alpha.
\end{multline*}
Applying the change of coordinates (\ref{invcoord}) yields
\begin{multline*}
\int|\tilde{L}_{l, k, j}(x)|\,dx\lesssim 
2^{k(1-\kappa_{\Omega})}\int_{I_j^{\ast}}\int_{\substack{|u_1|\le |I_j|^{-1}2^{-l}\\|u_2|\gtrsim 2^{-l}}}\frac{2^k}{(1+2^k|u_2|)^{3}}\,du_1\,du_2\,d\alpha
\\
\lesssim 2^{l-k}2^{-k\kappa_{\Omega}}.
\end{multline*}
By (\ref{newcard}), there are $\lesssim 2^{l/2}$ intervals $I_j$, so we may sum in $j$ and then in $k$ to obtain (\ref{tildaest}).
\end{proof}
To prove (\ref{kglest}), it remains to prove
\begin{lemma}\label{longlemma}
Let $S_{l, k, j, n}$ be as defined previously. Then
\begin{align}\label{notildaest}
\Norm{\big(\sum_{k: k\ge l}\sum_j\sum_{n: n>{k-l}}S_{l, k, j, n}(a_Q)}_{L^1(\mathbb{R}^2)}\lesssim 1.
\end{align}
\end{lemma}

Recall our treatment of the kernels $K_{k, j, n}$ in Section \ref{l1kerest}. In order to achieve sufficient decay in $n$ for $\int|K_{k, j, n}(x)|\,dx$ to prove an endpoint estimate, we would have had to integrate by parts twice in the $\alpha$ variable. However, doing so would make our estimates for $\int|K_{k, j, n}(x)|\,dx$ ultimately depend on the $C^2$ norm of the graph of $\partial\Omega$. Thus in our analysis of the kernels of the operators $S_{l, k, j, n}$, we will instead opt to approximate $\partial\Omega$ by a smooth curve whose curvature is essentially constant on ``Minkowski caps" of width $2^{-k}$, allowing us to perform the necessary integration by parts.
\newline
\indent 
Recall that $\{I_j\}=\{[b_j, b_{j+1}]\}$ is the partition of $[-1, 1]$ into subintervals with endpoints in $\tilde{\mathfrak{A}}(2^{-l})$, where $\tilde{\mathfrak{A}}(2^{-l})$ is the refinement of $\mathfrak{A}(2^{-l})$ given by Lemma \ref{endpointlemma}. Fix $k\ge l$, and let $\{J_m\}=\{[c_m, c_{m+1}]\}$ be the partition of $[-1, 1]$ into subintervals with endpoints in $\mathfrak{A}(2^{-k})$. We will prove the following approximation lemma.
\begin{lemma}\label{approxlemma2}
Fix integers $l, k\ge 0$ with $k\ge l$, and define $\{I_j\}$ and $\{J_m\}$ as above. Then there exists a smooth function $\gamma_k: [-1, 1]\to\mathbb{R}$ such that for every $x\in\mathfrak{A}(2^{-k})$,
\begin{align}\label{p1}
\gamma_k(x)=\gamma(x),
\end{align}
\begin{align}\label{p2}
\gamma_k^{\prime}(x)=\gamma^{\prime}(x),
\end{align}
and for every $\alpha\in J_m$,
\begin{align}\label{p3}
|\gamma_k^{\prime\prime}(\alpha)|\lesssim(\gamma^{\prime}(c_{m+1})-\gamma^{\prime}(c_m))|J_m|^{-1}\lesssim 2^{-k}|J_m|^{-2},
\end{align}
and 
\begin{align}\label{p4}
\int_{J_m}|\gamma_k^{\prime\prime\prime}(\alpha)|\,d\alpha\lesssim 2^{-k}|J_m|^{-2}.
\end{align}
Moreover, for every $j$,
\begin{align}\label{p7}
\int_{I_j^{\ast}}|I_j||\gamma_k^{\prime\prime}(\alpha)|\,d\alpha\lesssim 2^{-l}
\end{align}
and for any $\alpha\in I_j^{\ast}$,
\begin{align}\label{p8}
|\gamma_k^{\prime}(\alpha)-\gamma^{\prime}(\alpha)|\lesssim 2^{-l}|I_j|^{-1}.
\end{align}
\end{lemma}

\begin{remark}\label{randrem}
Note that (\ref{p2}) and (\ref{p3}) imply that for every $\alpha\in J_m$,
\begin{multline}\label{p5}
|\gamma(\alpha)-\gamma_k(\alpha)|\lesssim\int_{J_m}|\gamma^{\prime}(\alpha)-\gamma_k^{\prime}(\alpha)|\,d\alpha
\\
\lesssim\int_{J_m}\int_{c_m}^{\alpha}(|\gamma^{\prime\prime}(t)|+|\gamma_k^{\prime\prime}(t)|)\,dt\,d\alpha\lesssim (\gamma^{\prime}(c_{m+1})-\gamma^{\prime}(c_m))|J_m|\lesssim 2^{-k},
\end{multline}
and
\begin{multline}\label{p6}
|\gamma^{\prime}(\alpha)-\gamma_k^{\prime}(\alpha)|\lesssim \int_{J_m}(|\gamma^{\prime\prime}(\alpha)|+|\gamma_k^{\prime\prime}(\alpha)|)\,d\alpha
\\
\lesssim \gamma^{\prime}(c_{m+1})-\gamma^{\prime}(c_m)\lesssim 2^{-k}|J_m|^{-1}.
\end{multline}

\end{remark}

\begin{proof}[Proof of Lemma \ref{approxlemma2}]

The idea of the construction is to first define $\gamma_k$ near each point $x\in\mathfrak{A}(2^{-k})$ so that its graph is a line segment with slope $\gamma^{\prime}(x)$, to connect these line segments with curves of constant curvature, and then to smooth things out using an appropriate mollifier. We now proceed to give the details. 
\newline
\indent
We first define $\gamma_k$ in a neighborhood of each $x\in\mathfrak{A}(2^{-k})$. For each such $x$, let $J_{m(x)}$ be the element of $\{J_m\}$ whose right endpoint is $x$. Let $O_x$ be the interval $[x-\frac{|J_{m(x)}|}{100}, x+\frac{|J_{m(x)+1}|}{100}]$. Define a function $\gamma_{k, x}$ on $O_x$ so that $\{(\alpha, \gamma_{k, x}(\alpha)):\,\alpha\in O_x\}$ is the graph of a line segment satisfying $\gamma_{k, x}(x)=\gamma(x)$ and $\gamma_{k, x}^{\prime}(x)=\gamma^{\prime}(x)$. Let $x^+$ be the successor of $x$ in $x\in\mathfrak{A}(2^{-k})$. We now extend $\gamma_{k, x}$ to $\tilde{O}_x:=[x-\frac{|J_{m(x)}|}{100}, x^+-\frac{|J_{m(x)+1}|}{100}]$ by connecting the points 
\begin{align}\label{twopoints}
\bigg(x+\frac{|J_{m(x)+1}|}{100}, \gamma(x+\frac{|J_{m(x)+1}|}{100})\bigg); \bigg(x^+-\frac{|J_{m(x)+1}|}{100}, \gamma(x^+-\frac{|J_{m(x)+1}|}{100})\bigg)
\end{align} 
by the unique curve of constant curvature that has slope $\gamma^{\prime}(x)$ at the point 
\begin{align*}
\bigg(x+\frac{|J_{m(x)+1}|}{100}, \gamma_{k, x}(x+\frac{|J_{m(x)+1}|}{100})\bigg).
\end{align*} 
Note that for $\alpha$ between the two points (\ref{twopoints}), 
\begin{align}\label{t1}
|\gamma_{k, x}^{\prime\prime}(\alpha)|\lesssim (\gamma^{\prime}(c_{m+1})-\gamma^{\prime}(c_m))|J_m|^{-1}\lesssim 2^{-k}|J_{m(x)+1}|^{-2}.
\end{align}
Now define a piecewise smooth curve $\tilde{\gamma}_k:[-1, 1]\to\mathbb{R}$ by $\tilde{\gamma_k}|_{\tilde{O}_x}=\gamma_{k, x}$. 
\newline
\indent
For each $x\in\mathfrak{A}(2^{-k})$, let $U_x=[x+\frac{|J_{m(x)+1}|}{200}, x^+-\frac{|J_{m(x)+1}|}{200}]$. Let $\psi_{x}$ be a smooth positive bump function supported in 
\begin{align*}
\bigg[-\frac{|J_{m(x)+1}|}{800}, \frac{|J_{m(x)+1}|}{800}\bigg]
\end{align*} 
with $\int\psi_{x}=1$ and satisfying
\begin{align}\label{betaest}
D^{\beta}\psi_{x}\lesssim_{\beta} |J_{m(x)+1}|^{-\beta-1},\qquad\beta\ge 0\text{ an integer}.
\end{align}
Define a smooth curve $\gamma_k: [-1, 1]\to\mathbb{R}$ by $\gamma_k|_{U_x}=\tilde{\gamma}_k\ast\psi_{x}$ and $\gamma_k|_{(\bigcup_{x}U_{x})^c}=\tilde{\gamma}_k$. 
\newline
\indent
By construction, $\gamma_k$ satisfies (\ref{p1}) and (\ref{p2}). On $(\bigcup_xU_x)^c$, $\gamma_k^{\prime\prime}$ is identically $0$.  Let $\tilde{\gamma}_k^{\prime\prime}$ denote the a.e. defined pointwise second derivative of $\tilde{\gamma}_k$. Let $\tilde{\gamma}_{k, L}^{\prime}$ and $\tilde{\gamma}_{k, R}^{\prime}$ denote the (everywhere defined) left and right derivatives of $\tilde\gamma_k$, respectively. Then for $\alpha\in U_x\subset J_{m(x)+1}$, 
\begin{multline}\label{rand1}
|\gamma_k^{\prime\prime}(\alpha)|\lesssim|(\tilde{\gamma}_k^{\prime\prime}\ast\psi_x)(\alpha)|+|\tilde{\gamma}_{k, R}^{\prime}(x^+-\frac{|J_{m(x)+1}|}{100})-\tilde{\gamma}_{k, L}^{\prime}(x^+-\frac{|J_{m(x)+1}|}{100})|\Norm{\psi_x}_{\infty}
\\
\lesssim \sup_{\alpha\in U_x}|\tilde{\gamma}_k^{\prime\prime}(\alpha)|+(\gamma^{\prime}(c_{m+1})-\gamma^{\prime}(c_m))|J_m|^{-1}
\\
\lesssim (\gamma^{\prime}(c_{m+1})-\gamma^{\prime}(c_m))|J_m|^{-1}
\lesssim 2^{-k}|J_{m(x)+1}|^{-2}, 
\end{multline}
where in the second to last inequality we have used (\ref{t1}). Thus $\gamma_k$ satisfies (\ref{p3}). By (\ref{t1}) and (\ref{betaest}), we also have
\begin{multline*}
\int_{J_{m(x)+1}}|\gamma_k^{\prime\prime\prime}(\alpha)|\,d\alpha\lesssim\int_{U_x}|(\tilde{\gamma}_k^{\prime\prime}\ast\psi_x^{\prime})(\alpha)|\,d\alpha
\\
+|\tilde{\gamma}_{k, R}^{\prime}(x^+-\frac{|J_{m(x)+1}|}{100})-\tilde{\gamma}_{k, L}^{\prime}(x^+-\frac{|J_{m(x)+1}|}{100})|\Norm{\psi_x^{\prime}}_{\infty}|J_{m(x)+1}|
\\
\lesssim \int_{J_{m(x)+1}}2^{-k}|J_{m(x)+1}|^{-3}\,d\alpha+2^{-k}|J_{m(x)+1}|^{-2}\lesssim 2^{-k}|J_{m(x)+1}|^{-2},
\end{multline*}
and so $\gamma_k$ satisfies (\ref{p4}). 
\newline
\indent
Now we show that $\gamma_k$ satisfies (\ref{p7}). Note that (\ref{p3}) implies that for each $m$,
\begin{align}\label{tempy}
\int_{J_m}|\gamma_k^{\prime\prime}(\alpha)|\,d\alpha\lesssim\gamma^{\prime}(c_{m+1})-\gamma^{\prime}(c_m).
\end{align}
Given $I_j=[b_j, b_{j+1}]$, choose $m, m^{\prime}$ to the the greatest and least integers, respectively, so that $I_j^{\ast}\subset[c_{m}, c_{m^{\prime}}]$. Let $b_j^{\ast}$ and $b_{j+1}^{\ast}$ denote the left and right endpoints of $I_j^{\ast}$, respectively. If $b_j^{\ast}-c_m\le |I_j|/100$, then by (\ref{coomp}) we have $b_{j-1}\le c_m$, so by (\ref{tempy}) we have
\begin{align*}
\int_{I_j^{\ast}}|\gamma_k^{\prime\prime}(\alpha)|\,d\alpha\lesssim\gamma^{\prime}(c_{m^{\prime}})-\gamma^{\prime}(c_m)\lesssim\gamma^{\prime}(c_{m^{\prime}})-\gamma^{\prime}(b_{j-1}).
\end{align*}
Otherwise, $b_j^{\ast}-c_m >|I_j|/100$, and so (\ref{slope}) implies that
\begin{align*}
\gamma^{\prime}(b_j^{\ast})-\gamma^{\prime}(c_m)\lesssim 2^{-k}|I_j|^{-1}
\end{align*}
and hence
\begin{align*}
\int_{I_j^{\ast}}|\gamma_k^{\prime\prime}(\alpha)|\,d\alpha\lesssim \gamma^{\prime}(c_{m^{\prime}})-\gamma^{\prime}(c_m)\lesssim\gamma^{\prime}(c_{m^{\prime}})-\gamma^{\prime}(b_{j^{\ast}})+2^{-k}|I_j|^{-1}.
\end{align*}
In either case, we have
\begin{align*}
\int_{I_j^{\ast}}|\gamma_k^{\prime\prime}(\alpha)|\,d\alpha\lesssim\gamma^{\prime}(c_{m^{\prime}})-\gamma^{\prime}(b_{j-1})+2^{-k}|I_j|^{-1}.
\end{align*}
Arguing similarly with $c_{m^{\prime}}$ and $b_{j+1}^{\ast}$ in place of $c_m$ and $b_j^{\ast}$, we may obtain
\begin{align*}
\int_{I_j^{\ast}}|\gamma_k^{\prime\prime}(\alpha)|\,d\alpha\lesssim\gamma^{\prime}(b_{j+1})-\gamma^{\prime}(b_{j-1})+2^{-k}|I_j|^{-1}.
\end{align*}
By (\ref{slope}) and (\ref{coomp}), $\gamma^{\prime}(b_{j+1})-\gamma^{\prime}(b_{j-1})\lesssim 2^{-l}|I_j|^{-1}$, and since $k\ge l$ it follows that
\begin{align*}
\int_{I_j^{\ast}}|I_j||\gamma_k^{\prime\prime}(\alpha)|\,d\alpha\lesssim 2^{-l}.
\end{align*}
Thus $\gamma_k$ satisfies (\ref{p7}).
\newline
\indent
Finally, we show that $\gamma_k$ satisfies (\ref{p8}). Suppose we are given some $j$ and some $\alpha\in I_j^{\ast}$. If there exists $m$ such that $c_m\in I_j^{\ast}$, then by (\ref{p1}) and (\ref{p7}),
\begin{align*}
|\gamma_k^{\prime}(\alpha)-\gamma^{\prime}(\alpha)|\lesssim \int_{I_j^{\ast}}(|\gamma_k^{\prime\prime}(\alpha)|+|\gamma^{\prime\prime}(\alpha)|)\,d\alpha\lesssim 2^{-l}|I_j|^{-1}.
\end{align*}
Otherwise, choose $m$ so that the distance of $c_m$ from $I_j^{\ast}$ is minimal. Without loss of generality, suppose $c_m<b_{j}^{\ast}$. Then $c_{m+1}-c_m\gtrsim |I_j|$, so by (\ref{p3}) and (\ref{p7}),
\begin{multline*}
|\gamma_k^{\prime}(\alpha)-\gamma^{\prime}(\alpha)|\lesssim \int_{[c_{m}, c_{m+1}]\cup I_j^{\ast}}(|\gamma_k^{\prime\prime}(\alpha)|+|\gamma^{\prime\prime}(\alpha)|)\,d\alpha
\\
\lesssim 2^{-k}|I_j|^{-1}+2^{-l}|I_j|^{-1}\lesssim 2^{-l}|I_j|^{-1},
\end{multline*}
and hence $\gamma_k$ satisfies (\ref{p8}).
\end{proof}

\subsection*{The error estimate}
Define
\begin{multline*}
{B}_{l, k, j, n}(x)=\phi_0(2^{-6M}|x|)\int_0^{\infty}\int_{I_j^{\ast}}e^{is(\alpha x_1+\gamma_k(\alpha)x_2+1)}
\\
\beta_{I_j}(\alpha)\Phi_{k, j, n}(x, \alpha)\theta_k(s)a(s)s(\alpha\gamma_k^{\prime}(\alpha)-\gamma_k(\alpha))\,d\alpha\,ds.
\end{multline*}
Note that $B_{l, k, j, n}$ is like $L_{l, k , j, n}$ with every occurrence of $\gamma$ in the integral replaced by $\gamma_k$. We will prove
\begin{lemma}\label{errorlemma}
If $k\ge l$ and $n>k-l$, then
\begin{align}\label{diffest}
\Norm{L_{l, k, j, n}-B_{l, k, j, n}}_{L^1(\mathbb{R}^2)}\lesssim 2^{-k\kappa_{\Omega}}(2^{n-k}|I_j|^{-1}).
\end{align}
\end{lemma}

\begin{remark}\label{conseq}
We now state a consequence of Lemma (\ref{errorlemma}). By (\ref{newcard}), there are $\lesssim 2^{l\kappa_{\Omega}}$ intervals $I_j$. Moreover, the presence of $\phi_0(2^{-6M}|x|)$ implies that all terms with $2^{n-k}|I_j|^{-1}\gg 1$ are identically $0$, so (\ref{diffest}) implies that
\begin{align}\label{diffest2}
\sum_{\substack{k: k\ge l,\\ j,\\ n: n>k-l}}\Norm{L_{l, k, j, n}-B_{l, k, j, n}}_{L^1(\mathbb{R}^2)}\lesssim 1.
\end{align}
Then (\ref{diffest2}) implies that it suffices to prove (\ref{notildaest}) with $S_{l, k, j, n}$ replaced by the operator with kernel $B_{l, k, j, n}.$
\end{remark}

\begin{proof}[Proof of Lemma \ref{errorlemma}]
The first step is to write
\begin{align*}
L_{l, k, j, n}(x)-B_{l, k, j, n}(x)=H_1(x)+H_2(x),
\end{align*}
where
\begin{multline*}
H_1(x)=\phi_0(2^{-6M}|x|)\int_0^{\infty}\int_{I_j^{\ast}}e^{is(\alpha x_1+\gamma(\alpha)x_2+1)}(1-e^{is(\gamma_k(\alpha)x_2-\gamma(\alpha)x_2)})\\
\beta_{I_j}(\alpha)\Phi_{k, j, n}(x, \alpha)\theta_k(s)a(s)s(\alpha\gamma^{\prime}(\alpha)-\gamma(\alpha))\,d\alpha\,ds
\end{multline*}
and
\begin{multline*}
H_2(x)=\phi_0(2^{-6M}|x|)\int_0^{\infty}\int_{I_j^{\ast}}e^{is(\alpha x_1+\gamma_k(\alpha)x_2+1)}\beta_{I_j}(\alpha)\Phi_{k, j, n}(x, \alpha)
\\
\theta_k(s)a(s)s(\alpha(\gamma^{\prime}(\alpha)-\gamma_k^{\prime}(\alpha))-(\gamma(\alpha)-\gamma_k(\alpha)))\,d\alpha\,ds.
\end{multline*}
Note that the only places where the kernels $B_{l, k, j, n}$ and $L_{l, k, j, n}$ differ are in the complex exponential factor and the Jacobian factor in their integral representations. Here the term $H_1$ represents the difference in the complex exponential factor and the term $H_2$ represents the difference in the Jacobian factor. The estimation of $\int|H_1(x)|\,dx$ and $\int|H_2(x)|\,dx$ will share some similarities with the estimation of $\int |K_{k, j, n}(x)|\,dx$ from Section $3$. 

\subsubsection*{Estimation of $\int|H_1(x)|\,dx$}
We observe that (\ref{p5}) implies that for $s, x, \alpha$ in the support of $\phi_0(2^{-6M}|x|)\theta_k(s)\beta_{I_j}(\alpha)$ and for every integer $N\ge 0$,
\begin{align}\label{phase1}
|\partial_s^N\partial_{\alpha}(1-e^{is(\gamma_k(\alpha)x_2-\gamma(\alpha)x_2)})|\lesssim_N 2^{-kN}2^{k}|\gamma_k^{\prime}(\alpha)-\gamma^{\prime}(\alpha)||x|,
\end{align}
\begin{multline}\label{phase11}
|\partial_s^N\partial_{\alpha}^2(1-e^{is(\gamma_k(\alpha)x_2-\gamma(\alpha)x_2)})|\lesssim_N 
\\
2^{-kN}|x|\bigg(2^{2k}|\gamma_k^{\prime}(\alpha)-\gamma^{\prime}(\alpha)|^2+2^k(|\gamma_k^{\prime\prime}(\alpha)|+|\gamma^{\prime\prime}(\alpha)|)\bigg)
\end{multline}
and
\begin{align}\label{phase2}
|\partial_s^N(1-e^{is(\gamma_k(\alpha)x_2-\gamma(\alpha)x_2)})|\lesssim_N 2^{-kN}|x|.
\end{align}
Integrating by parts $H_1$ once in $\alpha$ yields
\begin{multline*}
H_1(x)=\phi_0(2^{-6M}|x|)\int_0^{\infty}\int_{I_j^{\ast}}e^{is(\alpha x_1+\gamma(\alpha)x_2+1)}
\\
\partial_{\alpha}g_{l, k, j, n}(x, s, \alpha)\theta_k(s)a(s)\,ds
\end{multline*}
where
\begin{align*}
g_{l, k, j, n}(x, s, \alpha)=\frac{(1-e^{is(\gamma_k(\alpha)x_2-\gamma(\alpha)x_2)})\beta_{I_j}(\alpha)\Phi_{k, j, n}(x, \alpha)(\alpha\gamma^{\prime}(\alpha)-\gamma(\alpha))}{x_1+x_2\gamma^{\prime}(\alpha)}.
\end{align*}
Now if $\partial_{\alpha}$ hits the term $(1-e^{is(\gamma_k(\alpha)x_2-\gamma(\alpha)x_2)})$, then we may integrate by parts again in $\alpha$, since no higher derivatives of $\gamma$ or $\gamma_k$ will appear. Thus we will further decompose
\begin{align*}
H_1(x)=H_{1, 1}(x)+H_{1, 2}(x),
\end{align*}
where
\begin{align*}
H_{1, 1}(x)=\phi_0(2^{-6M}|x|)\int_0^{\infty}\int_{I_j^{\ast}}e^{is(\alpha x_1+\gamma(\alpha)x_2+1)}
h_{l, k, j, n, 1}(x, s, \alpha)\theta_k(s)a(s)\,ds
\end{align*}
\begin{align*}
H_{1, 2}(x)=\phi_0(2^{-6M}|x|)\int_0^{\infty}\int_{I_j^{\ast}}e^{is(\alpha x_1+\gamma(\alpha)x_2+1)}
h_{l, k, j, n, 2}(x, s, \alpha)\theta_k(s)a(s)\,ds,
\end{align*}
and 
\begin{multline*}
h_{l, k, j, n, 1}(x, s, \alpha)=
\\(1-e^{is(\gamma_k(\alpha)x_2-\gamma(\alpha)x_2)})\partial_{\alpha}[\frac{\beta_{I_j}(\alpha)\Phi_{k, j, n}(x, \alpha)(\alpha\gamma^{\prime}(\alpha)-\gamma(\alpha))}{x_1+x_2\gamma^{\prime}(\alpha)}],
\end{multline*}
\begin{multline*}
h_{l, k, j, n, 2}(x, s, \alpha)=
\\\partial_{\alpha}\bigg[\frac{\partial_{\alpha}[1-e^{is(\gamma_k(\alpha)x_2-\gamma(\alpha)x_2)}]\beta_{I_j}(\alpha)\Phi_{k, j, n}(x, \alpha)(\alpha\gamma^{\prime}(\alpha)-\gamma(\alpha))}{s(x_1+x_2\gamma^{\prime}(\alpha))^2}\bigg].
\end{multline*}
Here we may think of $H_{1, 1}$ as representing the case when $\partial_{\alpha}$ does not hit the term $(1-e^{is(\gamma_k(\alpha)x_2-\gamma(\alpha)x_2)})$ when we integrate $H_1$ by parts with respect to $\alpha$, and $H_{1, 2}$ may be thought of as representing the case when $\partial_{\alpha}$ does hit $(1-e^{is(\gamma_k(\alpha)x_2-\gamma(\alpha)x_2)})$.
\subsubsection*{Estimation of $\int|H_{1, 1}(x)|\,dx$}
Observe that (\ref{phase2}) with $N=0$ implies that
\begin{align*}
|h_{l, k, j, n, 1}(x, s, \alpha)|\lesssim\frac{|\gamma^{\prime\prime}(\alpha)|(|I_j|2^{k-n}|x|+1)+|I_j|^{-1}}{|x_1+x_2\gamma^{\prime}(\alpha)|}|x|.
\end{align*}
Thus integrating by parts in $s$ three times and using (\ref{phase2}) and the change of coordinates (\ref{invcoord}) yields
\begin{multline*}
\int|H_{1, 1}(x)|\,dx\lesssim 2^{-k\kappa_{\Omega}}\int_{I_j^{\ast}}(|\gamma^{\prime\prime}(\alpha)|(|I_j|2^{k-n}+1)+|I_j|^{-1})
\\
\times\int_{|u_1|\approx 2^{n-k}|I_j|^{-1}}\frac{1}{|u_1|}\frac{2^k}{(1+2^k|u_2|)^3}|u|\,du\,d\alpha
\\
\lesssim 2^{-k\kappa_{\Omega}}2^{n-k}|I_j|^{-1}\int_{I_j^{\ast}}(|\gamma^{\prime\prime}(\alpha)|(|I_j|2^{k-n}+1)+|I_j|^{-1})\,d\alpha.
\end{multline*}
By (\ref{slope}) and (\ref{coomp}), we have
\begin{align*}
\int_{I_j^{\ast}}|\gamma^{\prime\prime}(\alpha)||I_j|\,d\alpha\lesssim 2^{-l},
\end{align*}
and so when $n>k-l$,
\begin{align}\label{Ay1est}
\int|H_{1, 1}(x)|\,dx\lesssim 2^{-k\kappa_{\Omega}}2^{n-k}|I_j|^{-1}(2^{k-l-n}+1)\lesssim 2^{-k\kappa_{\Omega}}2^{n-k}|I_j|^{-1}.
\end{align}
\subsubsection*{Estimation of $\int|H_{1, 2}(x)|\,dx$} 
Note that (\ref{phase1}) and (\ref{phase11}) with $N=0$ implies that
\begin{multline}\label{littlehest}
|h_{l, k, j, n, 2}(x, s, \alpha)|\lesssim
\\
\frac{|\gamma^{\prime\prime}(\alpha)|(|I_j|2^{k-n}|x|+1)+|I_j|^{-1}}{|x_1+x_2\gamma^{\prime}(\alpha)|}|x|\bigg(2^{k-n}|I_j||\gamma_k^{\prime}(\alpha)-\gamma^{\prime}(\alpha)|\bigg)
\\
+\frac{|x|}{|x_1+x_2\gamma^{\prime}(\alpha)|}\bigg(2^{2k-n}|I_j||\gamma_k^{\prime}(\alpha)-\gamma^{\prime}(\alpha)|^2\bigg)
\\
+\frac{|x|}{|x_1+x_2\gamma^{\prime}(\alpha)|}2^{k-n}|I_j|(|\gamma_k^{\prime\prime}(\alpha)|+|\gamma^{\prime\prime}(\alpha)|).
\end{multline}
Using (\ref{phase1}), (\ref{phase2}), (\ref{littlehest}) and the change of coordinates (\ref{invcoord}), we have
\begin{multline*}
\int|H_{1, 2}(x)|\,dx\lesssim 
\\
\bigg(2^{-k\kappa_{\Omega}}\int_{I_j^{\ast}}2^{k-n}|I_j||\gamma_k^{\prime}(\alpha)-\gamma^{\prime}(\alpha)|\big(|\gamma^{\prime\prime}(\alpha)|(|I_j|2^{k-n}+1)+|I_j|^{-1}\big)
\\
\times\int_{|u_1|\approx 2^{n-k}|I_j|^{-1}}\frac{1}{|u_1|}\frac{2^k}{(1+2^k|u_2|)^3}|u|\,du\,d\alpha\bigg)
\\
+\bigg(2^{-k\kappa_{\Omega}}\int_{I_j^{\ast}}2^{2k-n}|I_j||\gamma_k^{\prime}(\alpha)-\gamma^{\prime}(\alpha)|^2
\\
\times\int_{|u_1|\approx 2^{n-k}|I_j|^{-1}}\frac{1}{|u_1|}\frac{2^k}{(1+2^k|u_2|)^3}|u|\,du\,d\alpha\bigg)
\\
+\bigg(2^{-k\kappa_{\Omega}}\int_{I_j^{\ast}}2^{k-n}|I_j|(|\gamma_k^{\prime\prime}(\alpha)|+|\gamma^{\prime\prime}(\alpha)|)
\\
\times\int_{|u_1|\approx 2^{n-k}|I_j|^{-1}}\frac{1}{|u_1|}\frac{2^k}{(1+2^k|u_2|)^3}|u|\,du\,d\alpha\bigg),
\end{multline*}
and hence proceeding as in the estimation of $\int|H_{1, 1}(x)|\,dx$ we have
\begin{multline*}
\int|H_{1, 2}(x)|\,dx\lesssim
\\
\bigg(2^{-k\kappa_{\Omega}}2^{n-k}|I_j|^{-1}\int_{I_j^{\ast}}2^{k-n}|I_j||\gamma_k^{\prime}(\alpha)-\gamma^{\prime}(\alpha)||\gamma^{\prime\prime}(\alpha)|\big((|I_j|2^{k-n}+1)+|I_j|^{-1}\big)\,d\alpha\bigg)
\\
+\bigg(2^{-k\kappa_{\Omega}}2^{n-k}|I_j|^{-1}\int_{I_j^{\ast}}2^{2k-n}|I_j||\gamma_k^{\prime}(\alpha)-\gamma^{\prime}(\alpha)|^2\,d\alpha\bigg)
\\
+\bigg(2^{-k\kappa_{\Omega}}2^{n-k}|I_j|^{-1}\int_{I_j^{\ast}}2^{k-n}|I_j|(|\gamma_k^{\prime\prime}(\alpha)|+|\gamma^{\prime\prime}(\alpha)|)\,d\alpha\bigg).
\end{multline*}
Note that since $\{I_j\}$ satisfies (\ref{slope}) and (\ref{coomp}), we have
\begin{align*}
\int_{I_j^{\ast}}|I_j||\gamma^{\prime\prime}(\alpha)|\,d\alpha\lesssim 2^{-l}.
\end{align*}
As stated in (\ref{p7}), we also have
\begin{align*}
\int_{I_j^{\ast}}|I_j||\gamma_{k}^{\prime\prime}(\alpha)|\,d\alpha\lesssim 2^{-l}.
\end{align*}
Thus we have
\begin{multline*}
\int|H_{1, 2}(x)|\,dx\lesssim
\\
\bigg(2^{-k\kappa_{\Omega}}2^{n-k}|I_j|^{-1}\int_{I_j^{\ast}}2^{k-n}|I_j||\gamma_k^{\prime}(\alpha)-\gamma^{\prime}(\alpha)||\big(|\gamma^{\prime\prime}(\alpha)|(|I_j|2^{k-n}+1)+|I_j|^{-1}\big)\,d\alpha\bigg)
\\
+\bigg(2^{-k\kappa_{\Omega}}2^{n-k}|I_j|^{-1}\int_{I_j^{\ast}}2^{2k-n}|I_j||\gamma_k^{\prime}(\alpha)-\gamma^{\prime}(\alpha)|^2\,d\alpha\bigg)
\\
+2^{-k\kappa_{\Omega}}2^{n-k}|I_j|^{-1}2^{-n+(k-l)}.
\end{multline*}
Now we bound the integrals over $I_j^{\ast}$ by a sum of integrals over all the $J_m$ such that $J_m\cap I_j^{\ast}\ne\emptyset$ and use (\ref{p6}). We have
\begin{multline*}
\int|H_{1, 2}(x)|\,dx\lesssim
\\
2^{-k\kappa_{\Omega}}2^{n-k}|I_j|^{-1}\sum_{m:\,J_m\cap I_j^{\ast}\ne\emptyset}\bigg(\int_{J_m}(2^{-n}\frac{|I_j|}{|J_m|}(|\gamma^{\prime\prime}(\alpha)|(|I_j|2^{k-n}+1)+|I_j|^{-1})\,d\alpha
\\
+\int_{J_m}2^{-n}\frac{|I_j|}{|J_m|^2}\,d\alpha\bigg)+2^{-k\kappa_{\Omega}}2^{n-k}|I_j|^{-1}2^{-n+(k-l)}.
\end{multline*}
Using (\ref{slope}) gives
\begin{align*}
\int_{J_m}\bigg(|\gamma^{\prime\prime}(\alpha)|(|I_j|2^{k-n}+1)+|I_j|^{-1}\bigg)\,d\alpha\lesssim 2^{-n}\frac{|I_j|}{|J_m|}+\frac{|J_m|}{|I_j|}.
\end{align*}
Therefore
\begin{multline}\label{H121}
\int|H_{1, 2}(x)|\,dx\lesssim 
\\
2^{-k\kappa_{\Omega}}2^{n-k}|I_j|^{-1}\sum_{m:\,J_m\cap I_j^{\ast}\ne\emptyset}\bigg(2^{-2n}\frac{|{I}_j|^2}{|J_m|^2}+2^{-n}+2^{-n}\frac{|I_j|}{|J_m|}\bigg)
\\
+2^{-k\kappa_{\Omega}}2^{n-k}|I_j|^{-1}2^{-n+(k-l)}.
\end{multline}
\indent
We now proceed to bound (\ref{H121}). We will first show that for any $j$,
\begin{align}\label{csineq}
\text{card}(\{m:\,J_m\cap I_j^{\ast}\ne\emptyset\})\lesssim 1+ \text{card}(\{m:\,J_m\subset I_j^{\ast}\})\lesssim 2^{(k-l)/2}.
\end{align}
By Cauchy-Schwarz, (\ref{right}) and (\ref{left}),
\begin
{multline*}
\text{card}(\{m:\,J_m\subset I_j^{\ast}\})\le\sum_{\{m:\,J_m\subset I_j^{\ast}\}}2^{k/2}(c_{m+1}-c_m)^{1/2}(\gamma^{\prime}(c_{m+1})-\gamma^{\prime}(c_{m}))^{1/2}
\\
\le 2^{k/2}\bigg(\sum_{\{m:\,J_m\subset I_j^{\ast}\}}c_{m+1}-c_m\bigg)^{1/2}\bigg(\sum_{\{m:\,J_m\subset I_j^{\ast}\}}\gamma^{\prime}(c_{m+1})-\gamma^{\prime}(c_m)\bigg)^{1/2}
\\
\le 2^{k/2}(b_{j+1}-b_j)^{1/2}(\gamma^{\prime}(b_{j+1})-\gamma^{\prime}(b_j))^{1/2}\le 2^{(k-l)/2},
\end{multline*}
which proves (\ref{csineq}). Using (\ref{csineq}), we have 
\begin{align}\label{c1}
\sum_{m:\,J_m\cap I_j^{\ast}\ne\emptyset, |J_m|\ge \frac{|I_j|}{100}}\bigg(2^{-2n}\frac{|{I}_j|^2}{|J_m|^2}+2^{-n}+2^{-n}\frac{|I_j|}{|J_m|}\bigg)\lesssim 1
\end{align}
and
\begin{align}\label{c2}
\sum_{m:\,J_m\cap I_j^{\ast}\ne\emptyset}2^{-n}\lesssim 1.
\end{align}
If $J_m\cap I_j^{\ast}\ne\emptyset$ and $|J_m|<\frac{|I_j|}{100}$, then $J_m\subset I_{j-1}\cup I_j\cup I_{j+1}$. We will write $\Delta_{I_j}(\gamma^{\prime})$ in place of $\gamma^{\prime}(b_{j+2})-\gamma^{\prime}(b_{j-1})$. Similarly define $\Delta_{J_m}(\gamma^{\prime})=\gamma^{\prime}(c_{m+1})-\gamma^{\prime}(c_m)$. By (\ref{slope}), we have
\begin{align*}
|I_j|\lesssim 2^{-l}(\Delta_{I_j}(\gamma^{\prime}))^{-1}.
\end{align*}
By (\ref{left}) and (\ref{right}), we also have
\begin{align*}
|J_m|\approx 2^{-k}(\Delta_{J_m}(\gamma^{\prime}))^{-1}.
\end{align*}
We thus have
\begin{multline}\label{c3}
\sum_{m:\,J_m\cap I_j^{\ast}\ne\emptyset,\,|J_m|<|I_j|/100}\bigg(2^{-2n}\frac{|I_j|^2}{|J_m|^2}+2^{-n}\frac{|I_j|}{|J_m|}\bigg)
\\
\lesssim
\sum_{m:\,J_m\cap I_j^{\ast}\ne\emptyset,\,|J_m|<|I_j|/100}\bigg(2^{-2n}2^{2(k-l)}\bigg(\frac{\Delta_{J_m}(\gamma^{\prime})}{\Delta_{I_j}(\gamma^{\prime})}\bigg)^2+2^{-n}\bigg(\frac{\Delta_{J_m}(\gamma^{\prime}}{\Delta_{I_j}(\gamma^{\prime})}\bigg)\bigg)
\\
\lesssim 2^{-n+k-l}\lesssim 1.
\end{multline}
Together, (\ref{H121}), (\ref{c1}), (\ref{c2}) and (\ref{c3}) imply that when $n>k-l$ we have
\begin{align}\label{h12est}
\int |H_{1, 2}(x)|\,dx\lesssim 2^{-k\kappa_{\Omega}}2^{n-k}|I_j|^{-1}.
\end{align}
Together (\ref{Ay1est}) and (\ref{h12est}) imply that
\begin{align}\label{finalh1est}
\int |H_{1}(x)|\,dx\lesssim 2^{-k\kappa_{\Omega}}2^{n-k}|I_j|^{-1},
\end{align}
completing the estimation of $\int|H_{1}(x)|\,dx$.

\subsubsection*{Estimation of $\int|H_2(x)|\,dx$}
Integrating by parts $H_2$ once in $\alpha$ and twice in $s$ yields
\begin{align*}
\int|H_2(x)|\,dx\lesssim2^{-k\kappa_{\Omega}}\int\phi_0(2^{-6M}|x|)\int_{I_j^{\ast}}|\partial_{\alpha}g_{l, k, j, n}(x, \alpha)|
\\
\times\frac{2^k}{(1+2^k|\alpha x_1+\gamma(\alpha)x_2+1)|)^2}\,d\alpha\,dx,
\end{align*}
where
\begin{align*}
g_{l, k, j, n}(x, \alpha)=\frac{\Phi_{k, j, n}(x, \alpha)\beta_{I_j}(\alpha)[\alpha(\gamma^{\prime}(\alpha)-\gamma_k^{\prime}(\alpha))-(\gamma(\alpha)-\gamma_k(\alpha))]}{x_1+x_2\gamma^{\prime}(\alpha)}.
\end{align*}
By (\ref{p8}) and (\ref{p5}), for $\alpha$ in the support of $\beta_{I_j}(\alpha)$ we have
\begin{align}\label{jacob1}
|\alpha(\gamma^{\prime}(\alpha)-\gamma_k^{\prime}(\alpha))-(\gamma(\alpha)-\gamma_k(\alpha))|\lesssim 2^{-l}|I_j|^{-1}.
\end{align}
It is easy to see that (\ref{jacob1}) implies
\begin{align}\label{jacob2}
\bigg|\partial_{\alpha}\bigg[\alpha(\gamma^{\prime}(\alpha)-\gamma_k^{\prime}(\alpha))-(\gamma(\alpha)-\gamma_k(\alpha))\bigg]\bigg|
\lesssim 2^{-l}|I_j|^{-1}+|\gamma^{\prime\prime}(\alpha)|+|\gamma_k^{\prime\prime}(\alpha)|.
\end{align}
By (\ref{jacob1}) and (\ref{jacob2}), for $x$ in the support of $H_2$ we have
\begin{align*}
|\partial_{\alpha}g_{l, k, j, n}(x, \alpha)|\lesssim 2^{-l}|I_j|^{-1}\frac{(|\gamma^{\prime\prime}(\alpha)|+|\gamma_k^{\prime\prime}(\alpha)|)(|I_j|2^{k}+1)+|I_j|^{-1}}{|x_1+x_2\gamma^{\prime}(\alpha)|},
\end{align*}
and so applying the change of coordinates (\ref{invcoord}) and estimating the integral using (\ref{slope}) and (\ref{coomp}) as we did above in the estimation of $\int|H_1(x)|\,dx$, we obtain for $k\ge l$ and $n>k-l$,
\begin{align}\label{Ay2est}
\int|H_2(x)|\,dx\lesssim 2^{-k\kappa_{\Omega}}2^{-l}|I_j|^{-1}\lesssim 2^{-k\kappa_{\Omega}}2^{n-k}|I_j|^{-1}.
\end{align}
Together (\ref{finalh1est}) and (\ref{Ay2est}) imply that (\ref{diffest}) holds whenever $n>k-l$, completing the proof of the lemma.
\end{proof}

\subsection*{Estimation of the main term}
We have thus shown that to prove Lemma (\ref{longlemma}), it suffices to prove 
\begin{lemma}\label{maintermlemma}
Let $B_{l, k, j, n}$ be as defined previously. Then
\begin{align}\label{best}
\Norm{\big(\sum_{k: k\ge l}\sum_j\sum_{n: n>{k-l}}B_{l, k, j, n}(a_Q)}_{L^1(\mathbb{R}^2)}\lesssim 1.
\end{align}
\end{lemma}

\begin{proof}[Proof of Lemma \ref{maintermlemma}]

We have
\begin{multline*}
B_{l, k, j, n}(x)=\phi_0(2^{-6M}|x|)\int_0^{\infty}\int_{I_j^{\ast}}e^{is(\alpha x_1+\gamma_k(\alpha)x_2+1)}
\\
\times\beta_{I_j}(\alpha)\Phi_{k, j, n}(x, \alpha)\theta_k(s)a(s)s(\alpha\gamma_k^{\prime}(\alpha)-\gamma_k(\alpha))\,d\alpha\,ds.
\end{multline*}
We integrate by parts twice in $\alpha$ to obtain
\begin{multline*}
B_{l, k, j, n}(x)=\phi_0(2^{-6M}|x|)\int_0^{\infty}\int_{I_j^{\ast}}e^{is(\alpha x_1+\gamma_k(\alpha)x_2+1)}g_{l, k, j, n}(x, \alpha)
\\
\times s^{-1}\theta_k(s)a(s)\,d\alpha\,ds.
\end{multline*}
where
\begin{align*}
g_{l, k, j, n}(x, \alpha)=\partial_{\alpha}[\frac{1}{x_1+x_2\gamma_k^{\prime}(\alpha)}\partial_{\alpha}[\frac{\beta_{I_j}(\alpha)\Phi_{k, j, n}(x, \alpha)(\alpha\gamma_k^{\prime}(\alpha)-\gamma_k(\alpha))}{x_1+x_2\gamma_k^{\prime}(\alpha)}]].
\end{align*}
Integrating by parts twice in $s$ yields
\begin{multline*}
\int|B_{l, k, j, n}(x)|\,dx\lesssim 2^{-k(\kappa_{\Omega}+1)}\int\phi_0(2^{-6M}|x|)\int_{I_j^{\ast}}|g_{l, k, j, n}(x, \alpha)|
\\
\times\frac{2^k}{(1+2^k|\alpha x_1+\gamma_k(\alpha)x_2+1|)^{2}}\,d\alpha\,dx.
\end{multline*}
Observe that for $x$ in the support of $\phi_0(2^{-6M}|x|)$,
\begin{align*}
|g_{l, k, j, n}(x, \alpha)|\lesssim\frac{2^{k-n}|I_j||\gamma_k^{\prime\prime\prime}(\alpha)|+2^{2(k-n)}|I_j|^2|\gamma_k^{\prime\prime}(\alpha)|^2+|I_j|^{-2}}{|x_1+x_2\gamma_k^{\prime}(\alpha)|^2}.
\end{align*}
Thus using the change of coordinates 
\begin{align*}
(x_1, x_2)\mapsto(u_1, u_2):=(x_1+x_2\gamma_k^{\prime}(\alpha), 1+\alpha x_1+\gamma_k(\alpha)x_2),
\end{align*} 
we have
\begin{multline*}
\int|B_{l, k, j, n}(x)|\,dx\lesssim 2^{-k(\kappa_{\Omega}+1)}\int_{I_j^{\ast}}(2^{k-n}|I_j||\gamma_k^{\prime\prime\prime}(\alpha)|
\\
+2^{2(k-n)}|I_j|^2|\gamma_k^{\prime\prime}(\alpha)|^2+|I_j|^{-2})\int_{|u_1|\approx 2^{n-k}|I_j|^{-1}}\frac{1}{|u_1|^2}\frac{2^k}{(1+2^k|u_2|)^2}\,du\,d\alpha
\\
\lesssim 2^{-k(\kappa_{\Omega}+1)}2^{-n+k}|I_j|\int_{I_j^{\ast}}(2^{k-n}|I_j||\gamma_k^{\prime\prime\prime}(\alpha)|+2^{2(k-n)}|I_j|^2|\gamma_k^{\prime\prime}(\alpha)|^2+|I_j|^{-2})\,d\alpha.
\end{multline*}
Since 
\begin{align*}
2^{-k(\kappa_{\Omega}+1)}2^{-n+k}|I_j|\int_{|I_j|^{\ast}}|I_j|^{-2}\,d\alpha\lesssim 2^{-k\kappa_{\Omega}}2^{-n},
\end{align*}
we have
\begin{multline*}
\int|B_{l, k, j, n}(x)|\,dx\lesssim \bigg(2^{-k(\kappa_{\Omega}+1)}2^{-n+k}|I_j|\int_{I_j^{\ast}}(2^{k-n}|I_j||\gamma_k^{\prime\prime\prime}(\alpha)|
\\
+2^{2(k-n)}|I_j|^2|\gamma_k^{\prime\prime}(\alpha)|^2\,)d\alpha\bigg)+2^{-k\kappa_{\Omega}}2^{-n}.
\end{multline*}
Now for each $m$, choose $j(m)$ so that $I_{j(m)}^{\ast}\cap J_m\ne\emptyset$ and $I_{j(m)}$ has maximal length. Then using (\ref{newcard}), we have
\begin{multline*}
\sum_j\int|B_{l, k, j, n}(x)|\,dx\lesssim2^{-n}+2^{-n}2^{-k\kappa_{\Omega}}\sum_j\int_{I_j^{\ast}}2^{k-n}|I_j|^2|\gamma_k^{\prime\prime\prime}(\alpha)|\,d\alpha
\\
+2^{-n}2^{-k\kappa_{\Omega}}\sum_j\int_{I_j^{\ast}}2^{2(k-n)}|I_j|^3|\gamma_k^{\prime\prime}(\alpha)|^2\,d\alpha
\\
\lesssim 2^{-n}+2^{-n}2^{-k\kappa_{\Omega}}\sum_m2^{-n}\frac{|{I}_{j(m)}|^2}{|J_m|^2}\int_{{J_m}}2^{k}|J_m|^2|\gamma_k^{\prime\prime\prime}(\alpha)|\,d\alpha
\\
+2^{-n}2^{-k\kappa_{\Omega}}\sum_m2^{-2n}\frac{|{I}_{j(m)}|^3}{|J_m|^3}\int_{J_m}2^{2k}|J_m|^{3}|\gamma_k^{\prime\prime}(\alpha)|^2\,d\alpha.
\end{multline*}
Using (\ref{p3}) and (\ref{p4}), we have
\begin{align*}
\sum_j\int|B_{l, k, j, n}(x)|\,dx\lesssim 2^{-n}+2^{-n}2^{-k\kappa_{\Omega}}\sum_m(2^{-n}\frac{|{I}_{j(m)}|^2}{|J_m|^2}+2^{-2n}\frac{|{I}_{j(m)}|^3}{|J_m|^3}),
\end{align*}
and hence using that $n>k-l$,
\begin{multline*}
\sum_j\int|B_{l, k, j, n}(x)|\,dx\lesssim 
\\
2^{-n}+2^{2(k-l-n)}2^{-k\kappa_{\Omega}}\sum_m(2^{-2(k-l)}\frac{|{I}_{j(m)}|^2}{|J_m|^2}+2^{-3(k-l)}\frac{|{I}_{j(m)}|^3}{|J_m|^3}).
\end{multline*}
Since there are at most $\lesssim 2^{l\kappa_{\Omega}}$ intervals $J_m$ such that for some $j$, $J_m\cap I_j^{\ast}\ne\emptyset$ and $|J_m|\ge |I_j|/100$, we have
\begin{multline}\label{f1}
2^{2(k-l-n)}2^{-k\kappa_{\Omega}}\sum_{m:\,|J_m|\ge |I_{j(m)}|/100}(2^{-2(k-l)}\frac{|{I}_{j(m)}|^2}{|J_m|^2}+2^{-3(k-l)}\frac{|{I}_{j(m)}|^3}{|J_m|^3})
\\
\lesssim 2^{(l-k)\kappa_{\Omega}}2^{-n}.
\end{multline}
Note that if $|J_m|<|I_j|/100$, then $J_m\subset I_{j(m)-1}\cup I_{j(m)}\cup I_{j(m)+1}$. We will write $\Delta_{{I}_j}(\gamma_k^{\prime})$ in place of $|\gamma_k^{\prime}(b_{j+2})-\gamma_k^{\prime}(b_{j-1})|$. Similarly define $\Delta_{J_m}(\gamma_k^{\prime})=|\gamma_k^{\prime}(c_{m+1})-\gamma_k^{\prime}(c_m)|$. By (\ref{slope}), (\ref{coomp}) and (\ref{p8}), for every $j$ we have
\begin{align*}
|I_j|\lesssim 2^{-l}(\Delta_{{I}_j}(\gamma_k^{\prime}))^{-1}.
\end{align*}
Moreover, (\ref{left}) and (\ref{right}) also imply that for every $m$
\begin{align*}
|J_m|\approx 2^{-k}(\Delta_{J_m}(\gamma_k^{\prime}))^{-1}.
\end{align*}
It follows that 
\begin{align*}
\frac{|{I}_{j(m)}|}{|J_m|}\lesssim 2^{k-l}\frac{\Delta_{J_m}(\gamma_k^{\prime})}{\Delta_{{I}_{j(m)}}(\gamma_k^{\prime})},
\end{align*}
and hence
\begin{multline}\label{f2}
2^{2(k-l-n)}2^{-k\kappa_{\Omega}}\sum_{m:\,|J_m|< |I_{j(m)}|/100}(2^{-2(k-l)}\frac{|{I}_{j(m)}|^2}{|J_m|^2}+2^{-3(k-l)}\frac{|{I}_{j(m)}|^3}{|J_m|^3})
\\
\lesssim2^{2(k-l-n)}2^{-k\kappa_{\Omega}}\sum_{m:\,|J_m|< |I_{j(m)}|/100}\frac{\Delta_{J_m}(\gamma_k^{\prime})}{\Delta_{{I}_{j(m)}}(\gamma_k^{\prime})}\lesssim 2^{(l-k)\kappa_{\Omega}}2^{2(k-l-n)}.
\end{multline}
Together (\ref{f1}) and (\ref{f2}) imply that
\begin{align}\label{f3}
\sum_j\int|B_{l, k, j, n}(x)|\,dx\lesssim 2^{-n}+2^{(l-k)\kappa_{\Omega}}2^{(k-l-n)}.
\end{align}
Summing over $n>k-l$ and $k\ge l$ yields (\ref{notildaest}).

\end{proof}

\subsection*{The case $k<l$}
To prove Proposition \ref{offprop}, it remains to prove the following lemma.
\begin{lemma}\label{klemma}
Let $S_k$ be defined as previously. Then
\begin{align*}
\Norm{\sum_{k:\,k<l}S_k(a_Q)}_{L^1(\mathbb{R}^2)}\lesssim 1.
\end{align*}
\end{lemma}

\begin{proof}[Proof of Lemma \ref{klemma}]
We will need to exploit the cancellation of the atom. Since $\int a_Q=0$, we only need prove that for $k<l$,
\begin{align}\label{kll1}
\sup_{y, y^{\prime}\in Q}\int_{\mathbb{R}^2}|K_k(x-y)-K_k(x-y^{\prime})|\,dx\lesssim 2^{k-l}.
\end{align}
Now,
\begin{multline*}
\sup_{y, y^{\prime}\in Q}\int_{\mathbb{R}^2}|K_k(x-y)-K_k(x-y^{\prime})|\,dx\lesssim\int\sup_{y, y^{\prime}\in Q}|K_k(x-y)-K_k(x-y^{\prime})|\,dx
\\
\lesssim2^{-l}\int\sup_{y\in Q}|\nabla K_k(x-y)|\,dx,
\end{multline*}
so to prove (\ref{kll1}) it suffices to show that
\begin{align}\label{kll2}
\int\sup_{y\in Q}|\nabla K_k(x-y)|\,dx\lesssim 2^k.
\end{align}
Since $k<l$ and since $(\nabla K_k)(x)=(K_k(\cdot)\ast2^{3k}\phi(2^k\cdot))(x)$ for some Schwartz function $\phi$, 
it is easy to see that 
\begin{align*}
\int\sup_{y\in Q}|\nabla K_k(x-y)|\,dx\lesssim 2^k\int|K_k(x)|\,dx.
\end{align*}
But by the proof of (\ref{kglest}) in the case that $k=l$ and the estimation of the term $K_{k, j, 0}$ from Section \ref{l1kerest}, we have
\begin{align*}
\int|K_k(x)|\,dx\lesssim 1,
\end{align*}
which implies (\ref{kll2}) and finishes the proof.
\end{proof}

\section{Estimates for a generalized Bochner-Riesz square function}\label{bochnersec}
In \cite{cgt}, Carbery, Gasper and Trebels showed that one may use the sharp $L^4$ estimates for the two-dimensional Bochner-Riesz square function, first obtained by Carbery in \cite{car}, to prove multiplier theorems for radial Fourier multipliers in $\mathbb{R}^2$. We are thus motivated to consider the generalized Bochner-Riesz square function
\begin{align*}
G^{\alpha}f(x)=\bigg(\int_0^{\infty}\bigg|\frac{\partial}{\partial t}\mathcal{R}_t^{\alpha}f(x)\bigg|^2\,t\,dt\bigg)^{1/2}.
\end{align*}
In the same vein as in \cite{cgt}, $L^4$ estimates for $G^{\alpha}$ yield a multiplier theorem for quasiradial multipliers in the range $4/3\le p\le4$, which we will then interpolate with Theorem \ref{multscalethm}. In \cite{clad}, the following $L^4$ estimate for $G^{\alpha}$ is obtained.
\begin{proposition}\label{sqrprop}
For $\alpha>-1/2$,
\begin{align*}
\Norm{G^{\alpha}f}_4\lesssim_M\Norm{f}_4.
\end{align*}
\end{proposition}
Following \cite{cgt}, one may then obtain the following corollary. 
\begin{corollary}\label{rieszcor}
If $\alpha>1/2$, then for $4/3\le p\le 4$, 
\begin{align*}
\Norm{m\circ\rho}_{M^p(\mathbb{R}^2)}\lesssim\sup_{t>0}\bigg(\int
|\mathcal{F}_{\mathbb{R}}[\phi(\cdot)m(t\cdot)](\tau)|^2|\tau|^{2\alpha}\,d\tau\bigg)^{1/2}.
\end{align*}
\end{corollary}

\section{An interpolation argument}\label{anasec}
We now prove Theorem \ref{mainmultthm} by interpolating Corollary \ref{rieszcor} and Theorem \ref{multscalethm}.

\begin{proof}[Proof of Theorem \ref{mainmultthm}.]
Let $\tilde{\mathcal{S}}(\mathbb{R})$ denote the space of Schwartz functions on $\mathbb{R}$ with support in the annulus $\{x:\,1/2<|x|<2\}$. For $s\ge 0$ and $1\le r\le 2$ define norms $\Norm{\cdot}_r^s$ by
\begin{align*}
\Norm{f}_r^s=\bigg(\int|\widehat f(\tau)|^r(1+|\tau|)^{rs}\,d\tau\bigg)^{1/r},
\end{align*}
and let $L_{r}^s$ denote the space of all measurable functions $f$ with $\Norm{f}_r^s<\infty$. Let $\tilde{L}_r^s(\mathbb{R})$ denote the closure of $\tilde{\mathcal{S}}(\mathbb{R})$ in $L_r^s(\mathbb{R})$. For each integer $N\ge 0$, let  $C_{0, N}$ denote the space of sequences with support in $[-N, N]$, and let $\ell^{\infty}_N$ denote the closure of $C_{0, N}$ in $\ell^{\infty}$. For $N\in\mathbb{N}$, define a bilinear operator $T_{N}$ where $T_{N}: \mathcal{S}(\mathbb{R}^2)\times C_{0, N}(\tilde{S}(\mathbb{R}))\to \mathcal{S}(\mathbb{R}^2)$ by 
\begin{align*}
\mathcal{F}[T_{N}(f, \{m_k\}_{k=-N}^N)(\cdot)](\xi)=\sum_{k=-N}^Nm_k(2^{-k}\rho(\xi))\hat{f}(\xi).
\end{align*}
Then Theorem \ref{multscalethm} implies that for $s>\kappa_{\Omega}$ and for every $N$ and $1<p<\infty$, $T_{N}$ extends to a bounded bilinear operator from $L^p(\mathbb{R}^2)\times \ell^{\infty}_N(\tilde{L}_1^s(\mathbb{R}))$ to $L^p(\mathbb{R}^2)$ with operator norm
\begin{align}\label{bilinear1}
\Norm{T_{N}}_{L^p(\mathbb{R}^2)\times \ell^{\infty}_N(\tilde{L}_1^s(\mathbb{R}))\to L^p(\mathbb{R}^2)}=C_{p, s}
\end{align}
for some constant $C_{p}>0$ depending only on $p$ and $s$ and not on $N$. Corollary \ref{rieszcor} implies that for every $\alpha>1/2$ and for every $N$, $T_{N}$ extends to a bounded bilinear operator from $L^{4/3}(\mathbb{R}^2)\times \ell^{\infty}_N(\tilde{L}_2^{\alpha}(\mathbb{R}))$ to $L^{4/3}(\mathbb{R}^2)$ with operator norm
\begin{align}\label{bilinear2}
\Norm{T_{N}}_{L^{4/3}(\mathbb{R}^2)\times \ell^{\infty}_N(\tilde{L}_2^{\alpha}(\mathbb{R}))\to L^{4/3}(\mathbb{R}^2)}=C^{\prime}_{\alpha}
\end{align}
for some constant $C^{\prime}_{\alpha}>0$ depending only on $\alpha$ and not on $N$. Applying bilinear real interpolation methods (see for example \cite{bl}) to (\ref{bilinear1}) and (\ref{bilinear2}), we obtain for $0\le\theta\le 1$,
\begin{align}\label{bilinear3}
\Norm{T_{N}}_{L^{q_{0}}(\mathbb{R}^2)\times \ell^{\infty}_N(\tilde{L}_{q_1}^{s_0(\epsilon)}(\mathbb{R}))\to L^{q_0}(\mathbb{R}^2)}\lesssim_{\epsilon, p, \theta}1,
\end{align}
where
\begin{align}\label{exp}
\frac{1}{q_0}=\frac{1-\theta}{p}+\frac{\theta}{4/3}, \qquad \frac{1}{q_1}=1-\frac{\theta}{2}, \qquad s_0(\epsilon)=(1-\theta)\kappa_{\Omega}+\frac{\theta}{2}+\epsilon.
\end{align}
Define a bilinear operator $T: \mathcal{S}(\mathbb{R}^2)\times\ell^{\infty}(\tilde{L}_1^{0}(\mathbb{R}))\to L^2(\mathbb{R}^2)$ by
\begin{align*}
\mathcal{F}[T(f, \{m_k\}_{k=-\infty}^{\infty})(\cdot)](\xi)=\sum_{k=-\infty}^{\infty}m_k(2^{-k}\rho(\xi))\widehat f(\xi).
\end{align*}
Using (\ref{bilinear3}) and letting $N\to\infty$, we obtain
\begin{align*}
\Norm{T}_{L^{q_{0}}(\mathbb{R}^2)\times \ell^{\infty}(\tilde{L}_{q_1}^{s_0(\epsilon)}(\mathbb{R}))\to L^{q_0}(\mathbb{R}^2)}\lesssim_{\epsilon, p, \theta}1,
\end{align*}
for $q_0, q_1, s_0(\epsilon)$ as in (\ref{exp}). Set $s(\kappa_{\Omega}, \theta)=(1-\theta)\kappa_{\Omega}+\frac{\theta}{2}$. Since $1<p<\infty$, we have
\begin{align}\label{finalest}
\Norm{T}_{L^{q_{0}}(\mathbb{R}^2)\times \ell^{\infty}(\tilde{L}_{\frac{2}{2-\theta}}^{s(\kappa_{\Omega}, \theta)+\epsilon}(\mathbb{R}))\to L^{q_0}(\mathbb{R}^2)}\lesssim_{\epsilon, q_0, \theta}1,
\end{align}
for any $\frac{4}{4-\theta}<q_0<\frac{4}{\theta}$.
It is straightforward to see that (\ref{finalest}) implies the result. 
\end{proof}

\end{document}